\documentclass[11pt]{amsart}
\usepackage{amssymb}
\usepackage{amsmath}
\usepackage{mathrsfs}
\usepackage{faktor}
\usepackage{tikz}
\usepackage{thmtools}
\usepackage{mathrsfs,amsmath, tikz,verbatim, wasysym,xcolor, amsfonts, amssymb, amsthm, graphicx, hyperref, enumerate, xfrac, xspace, mathtools}
\usepackage{thm-restate}
\usetikzlibrary{arrows}
\usetikzlibrary{shapes.geometric}
\usetikzlibrary{matrix}
\usetikzlibrary{decorations.pathreplacing, decorations.markings}
\usepackage[utf8]{inputenc}
\usepackage[top=1in, bottom=1in, right=1in, left=1in]{geometry}
\usepackage{hyperref}
\usepackage{graphicx}
\usepackage{xcolor}
\usepackage{tikz-cd}

\usepackage[capitalise]{cleveref}

\newtheorem{theorem}{Theorem}[section]
\newtheorem{lem}[theorem]{Lemma}
\newtheorem{proposition}[theorem]{Proposition}
\newtheorem{cor}[theorem]{Corollary}

\theoremstyle{definition}
\newtheorem{dfn}[theorem]{Definition}
\newtheorem{ex}[theorem]{Example}

\newtheorem{claim}[theorem]{Claim}

\numberwithin{theorem}{section}
\allowdisplaybreaks

\newenvironment{theorem_no_number}[1][]{\begin{trivlist}
\item[\hskip \labelsep {\bfseries Theorem \def\temp{#1}\ifx\temp\empty  #1\else  #1\fi
.}] \itshape}  {\end{trivlist}}
\numberwithin{theorem}{section}

\DeclareMathOperator{\id}{id}

\DeclareMathOperator{\im}{im}

\DeclareMathOperator{\dom}{dom}

\DeclareMathOperator{\Dom}{\delta}

\DeclareMathOperator{\IP}{IP}

\DeclareMathOperator{\Path}{Path}
\newcommand{\s}{\mathbf{s}}
\newcommand{\rr}{\mathbf{r}}
\newcommand{\Z}{\mathbb{Z}}
\newcommand{\N}{\mathbb{N}}

\newcommand{\InvPres}[2]{\mathrm{Inv}\langle #1 \mid #2 \rangle}

\newenvironment{AL}{\noindent\color{red} AL:}{}

\newcommand{\makeset}[2]{\left\lbrace #1 \;\middle|\;
  \begin{tabular}{@{}l@{}}
    #2
   \end{tabular}
  \right\rbrace}
\newcommand{\set}[2]{\{ #1 \,\mid\, #2\}}
\title[\texorpdfstring{$E$}{E}-disjunctive inverse semigroups]{\texorpdfstring{$\boldsymbol E$}{E}-disjunctive inverse semigroups}
\newcommand{\downset}{\hspace{-2pt} \downarrow}

\usepackage[backend=biber, bibencoding=utf8, giveninits=true,
maxbibnames=99]{biblatex}

\author{Luna Elliott, Alex Levine and James Mitchell}


\begin{document}
\begin{abstract}
  In this paper we provide an overview of the class of inverse semigroups $S$ such that every congruence on $S$ relates at least one idempotent to a non-idempotent; such inverse semigroups are called \textit{$E$-disjunctive}. This overview includes the study of the inverse semigroup theoretic structure of $E$-disjunctive semigroups; a large number of natural examples; some asymptotic results establishing the rarity of such inverse semigroups; and a general structure theorem for all inverse semigroups where the building blocks are $E$-disjunctive.
\end{abstract}
\maketitle
\part{In the beginning}
\section{Introduction}

In this paper, we are concerned with a natural type of inverse semigroup. Recall
that a \textit{semigroup} is just a set with an associative binary operation,
and an \textit{inverse semigroup} is a semigroup $S$ where for every $x\in S$
there exists a unique $x^{-1} \in S$ such that $x x^{-1} x=x$ and $x^{-1} x
x^{-1} = x^{-1}$. Inverse semigroups have been extensively studied in the
literature since their inception. Roughly speaking, the class of inverse
semigroups lies somewhere between the classes of semigroups and groups, having
more structure, in general, than semigroups, and somewhat less structure than
groups. If $S$ is a semigroup, then an equivalence relation $\rho\subseteq S \times S$ is a
\textit{congruence} if whenever $(x, y)\in \rho$ and $s\in S$, it follows that
$(xs, ys), (sx, sy)\in \rho$ also. Congruences are to semigroups what normal
subgroups are to groups. In this paper, we are interested in the class of inverse
semigroups $S$ such that every congruence on $S$ relates at least one idempotent
to a non-idempotent element of $S$. Such inverse semigroups are called
\textit{\(E\)-disjunctive}; see~\cite[III.4]{petrich_book} for further information.

In this paper, we study the inverse semigroup theoretic structure of
$E$-disjunctive semigroups; give a large number of natural examples; give some
asymptotic results establishing the rarity of such inverse semigroups; and 
prove a general structure theorem for all inverse semigroups which can be built
from $E$-disjunctive inverse semigroups.

A congruence is called \textit{idempotent-pure} if it never relates an
idempotent to a non-idempotent. Idempotent-pure congruences have received
significant attention since the study of inverse semigroups commenced; see, for
example, \cites{Clifford_IP, Bernd, Kambites2015,
margolis_meakin, OCarroll1976, OCarroll75, PetrichReilly82}. Introduced by Green~\cite{green_cong},
they preserve much of the important structure of inverse semigroups and have
multiple equivalent definitions of different flavours. If $\rho$ is a congruence on an inverse semigroup \(S\), then the \emph{kernel} of $\rho$ is the (normal) inverse subsemigroup of $S$ consisting of the congruence classes of the idempotents, and the \emph{trace} is the restriction of the congruence to the semilattice of idempotents.
Conversely, distinct congruences have distinct kernel-trace pairs (see \cite[Section 5.3]{Howie}).
Hence
idempotent-pure congruences are those with trivial kernel, and thus are entirely
determined by their restriction to the idempotents of a given inverse semigroup.

\textit{\(E\)-disjunctive} inverse semigroups are those with no non-trivial
idempotent-pure congruences. Every inverse semigroup has an \(E\)-disjunctive
quotient by its syntactic congruence on its idempotents. It is not difficult to
show that the symmetric inverse monoid is \(E\)-disjunctive, and so every
inverse semigroup embeds into an \(E\)-disjunctive inverse semigroup. Slightly
more non-trivial is the proof that every inverse semigroup occurs as the
homomorphic image of an \(E\)-disjunctive inverse semigroup
(\cref{cor:all-qts}), thus showing that this class captures a large variety of
inverse semigroups. 

There are a modest number of papers in the literature about \(E\)-disjunctive
inverse semigroups. The first use of the term that we know of is in Petrich
\cite{petrich_book} from 1984. Shortly after in 1985,
Yoshida~\cite{yoshida_E_dis} published a short note on \(E\)-disjunctive inverse
semigroups, where it is shown that the class of \(E\)-disjunctive inverse
semigroups is closed under passing to full inverse subsemigroups; and an
alternative definition of \(E\)-disjunctivity was given. Yoshida also noted that
an earlier work of Alimpić and Krgović \cite{Clifford_IP} fully classifies when
a Clifford inverse semigroup is \(E\)-disjunctive through the description of
idempotent-pure homomorphisms. Additional classifications of \(E\)-disjunctivity
were provided by Li and Zhang \cite{Li_Zhang_E_dis}. Petrich and Reilly
\cite{petrich_reilly1}, and Gigon \cite{Gigon_E_dis} have also studied
\(E\)-disjunctivity in the non-inverse case.

This paper has three parts. In the first, we cover the basic properties of
$E$-disjunctive inverse semigroups, and their interactions with the standard
notions related to inverse semigroups (\cref{sec:basic_props}). These standard
notions include: the natural partial order; adjoining identities and zeros; and
basic closure properties such as direct products (\cref{sec:npo10}). In
\cref{sec:wreath-prod} we describe some circumstances under which wreath
products are $E$-disjunctive (\cref{prop:first_wreath},
\cref{thm_final_wreath}).

In the second part, we consider a compendium of examples of naturally occurring
$E$-disjunctive semigroups. These include the symmetric inverse monoids $I_X$ on
any set $X$ with at least 2 elements (see \cite[Section 5.1]{Howie} or the start
of \cref{section-compendium} for the definition, and
\cref{ex:symmetric_inverse_monoid} for the proof of $E$-disjunctivity); the dual
symmetric inverse monoids (\cref{section-compendium} and
\cref{ex:dual_symmetric_inverse_monoid}); some minimal examples of
$E$-disjunctive semigroups with certain properties (\cref{ex:GAP}); an infinite
finitely generated Thompson's group-like $E$-disjunctive inverse monoid
(\cite{CannonFloydParry} and \cref{ex:thompsonsV}); a proof that the arithmetic
inverse monoid from~\cite{Hines} is $E$-disjunctive in \cref{ex:AIM}. Graph
inverse semigroups arise naturally from the study of Leavitt path algebras. 
Such semigroups have been studied extensively in the literature in recent
years, see for example \cite{AnagnostopoulouMerkouri2024, JonesLawson, LuoWang, LuoWangWei, MeakinWang, 
MesyanMitchellMoraynePeresse,  MesyanMitchell,
Wang2019}. In \cref{section-gis}, we characterise the
idempotent-pure congruences on graph inverse semigroups in
\cref{thm:wang-triple-ip}, and characterise graph inverse semigroups that are
$E$-disjunctive in terms of the underlying graphs in \cref{Thm:graph_inverse}.
In the final section of this part of the paper, we characterise the finite
monogenic $E$-disjunctive inverse semigroups
(\cref{subsection-q-semigroups-are-semigroups}); and use this to show that the
number of monogenic $E$-disjunctive inverse semigroups as a proportion of all
monogenic inverse semigroups of order $n$ is asymptotically $0$
(\cref{cor:monogenic-proportion}).

In the third and final part of the paper we consider various structural
properties of $E$-disjunctive semigroups. In \cref{sec:ratio-idempotents}, we
show that there are fairly restrictive bounds on the number of idempotents and
non-idempotent elements in finite $E$-disjunctive semigroups
(\cref{idem_bound_thm}). We explore the extent to which information about an
arbitrary inverse semigroup can be recovered from its maximal $E$-disjunctive
image in \cref{sec:maximage}. In \cref{sec:mcealister}, we reprove a theorem
from~\cite{OCarroll75}\footnote{The authors of the present paper only discovered
\cite{OCarroll75} at a late stage of the preparation of this paper and prove
the characterization independently. The theorem and its proof are included for
the sake of completeness} which provides a means of constructing any inverse
semigroup from an $E$-disjunctive semigroup acting on a partially ordered set
(\cref{thm:Q-thm}). This theorem implies McAlister's famous
$P$-theorem from~\cite{McAlister} which characterises the $E$-unitary inverse
semigroups via groups acting on partially ordered sets. 

  \begin{table}
    \centering
   \begin{tabular}{l||r|r|r|r|r|r}
     $n$ & inverse & \(E\)-unitary  & $E$-disjunctive & \(E\)-disjunctive \\ 
      & semigroups \cite{OEIS} & (non-semilattice) & inverse semigroups & inverse monoids \\ \hline
     0   & 1    & 0        & 1 &  0\\
     1   & 1    & 0        & 1 &  1\\
     2   & 2    & 1        & 1 &  1\\
     3   & 5    & 2        & 2 &  2\\
     4   & 16   & 6       & 4  &  4\\
     5   & 52   & 12       & 8 &  6\\
     6   & 208  & 39       & 18&  15\\
     7   & 911     & 120     & 40 & 28\\
     8   & 4,637   & 483     & 101 & 68\\
     9   & 26,422  & 2,153     & 276 & 165 \\
     10  & 169,163 & 11,325     & 761 & 414 \\
     11  &1,198,651 &   67,570      &2,422 & 1,202 \\
     12  &9,324,047 &   453,698     &7,630             & 3,458  
   \end{tabular} \medskip
   
    \caption{Numbers of isomorphisms types of inverse semigroups of order \(n\)
    with certain properties; computed using the \textsf{GAP} package
    \textsf{Semigroups} \cite{Mitchell2023aa}, and \cite{MalandroNonLatticesPage,
    MalandroLatticesPage, Malandro, InverseSemigroupsWebpage}.}
    \label{tab:numbers}
\end{table}

\section{Basic Properties}
\label{sec:basic_props} 
In this section we give some of the basic properties of \(E\)-disjunctive
inverse semigroups. We also show how to construct new examples from old: via ideals (\cref{ideal_lem}); full subsemigroups
(\cref{full-subsemi}); direct products
(\cref{prop:dir_prod}); adjoining a zero or identity (\cref{cor:adj-1} and
 \cref{cor:add_zero}); zero direct unions (\cref{prop:0-direct}); and wreath
 products (\cref{thm_final_wreath}). We will be using \(E(S)\) to denote the set of idempotents
in an inverse semigroup \(S\).

A congruence \(\rho\) on a semigroup \(S\) is called \textit{idempotent-pure} if
\((s, e) \in \rho\) and \(e \in E(S)\) implies that \(s \in E(S)\).

  \begin{dfn}[\textbf{\(E\)-disjunctive.}]
    An inverse semigroup \(S\) is called \textit{\(E\)-disjunctive} if the only 
    idempotent-pure congruence on \(S\) is the trivial congruence $\Delta_S$.
  \end{dfn}

The numbers of $E$-disjunctive inverse semigroups of size $n$ for some small
values of $n$ are shown in \cref{tab:numbers}.

  \begin{ex}
    \label{ex:eunitary}
    Every group is an \(E\)-disjunctive inverse semigroup
    and symmetric inverse monoids on a set \(X\) are \(E\)-disjunctive if and only if $|X| \neq 1$; see \cref{section-compendium}.
    The free inverse monoids and the bicyclic monoid defined by
    the presentation $\langle b, c\mid bc = 1\rangle$ are not \(E\)-disjunctive; see \cref{section-gis}
    for more details.
  \end{ex}
  
  A useful tool when studying \(E\)-disjunctive inverse semigroups is the
  syntactic congruence with respect to the set of idempotents. This is the
  maximum idempotent-pure congruence on any inverse semigroup, and so will be
  trivial if and only if the semigroup is \(E\)-disjunctive.

  \begin{dfn}[\textbf{Syntactic congruence.}]
    Let \(S\) be an inverse semigroup. The \textit{syntactic congruence} (with
    respect to \(E(S)\))
    \(\rho\) on \(S\) is defined by \((s, \ t) \in \rho\) if and only if
    \[
      \alpha s \beta \in E(S) \quad\text{if and only if}\quad \alpha t \beta \in E(S),
    \]
    for all \(\alpha, \ \beta \in S^1\) where $S^1$ is the monoid obtained by
    adjoining an identity to $S$.

    Since the syntactic congruence with respect to \(E(S)\) is the only
    syntactic congruence we will be using, we will use the term ``syntactic
    congruence'' to mean this exclusively.
  \end{dfn}

The following lemma is well-known, we include a proof for completeness.

\begin{lem}\label{lem:largest-ip-cong}
    If \(S\) is an inverse semigroup, then the syntactic congruence \(\rho\) on
    \(S\) is the largest idempotent-pure congruence on \(S\) with respect to containment.
\end{lem}
\begin{proof}
    We first check that \(\rho\) is an idempotent-pure congruence. It is
    immediate from the definition that \(\rho\) is both a right and left
    congruence, hence \(\rho\) is a congruence. Suppose that \(e\in E(S)\) and
    \((e, s)\in \rho\).    Then \(1e1\in E(S)\), so by the definition of
    \(\rho\), \(s=1s1\in E(S)\).

    Let \(\tau\) be an idempotent-pure congruence on \(S\), and suppose that
    \((s, t)\in \tau\). Let \(\alpha, \beta\in S^1\). Then \[(\alpha s
    \beta,\alpha t \beta) \in \tau \] so as \(\tau\) is idempotent-pure,
    \(\alpha s \beta\in E(S)\) if and only if \(\alpha t \beta \in E(S)\). Hence
    \((s, t)\in \rho\).    
\end{proof}

The next lemma relates $E$-disjunctivity and the syntactic congruence.

  \begin{lem}[cf. Remark III.4.15($\delta$) in \cite{petrich_book}]
    \label{lem:syn-eq}
    Let \(S\) be an inverse semigroup. Then \(S\) is \(E\)-disjunctive if and
    only if the syntactic congruence is equality.
\end{lem}

The next result establishes that every inverse semigroup has an $E$-disjunctive
quotient. 

  \begin{lem}
    If $S$ is any inverse semigroup, then the quotient of $S$ by the syntactic congruence
    (which is idempotent-pure) is $E$-disjunctive. 
  \end{lem}

    The following lemma provides an alternative means of showing that inverse semigroups
    are \(E\)-disjunctive to computing the syntactic congruence. 

  \begin{lem}
      \label{lem-idempotents}
      Let \(S\) be an inverse semigroup. If every idempotent-pure congruence \(\rho\)
      is trivial on \(E(S)\); that is, for all \(e \in E(S)\) the congruence class of \(e\) is \(\{e\}\),
      then \(S\) is \(E\)-disjunctive.
  \end{lem}

  \begin{proof}
      Let \(\rho\) be an idempotent-pure congruence on \(S\). As \(\rho\) is idempotent-pure, the kernel of \(\rho\) is \(E(S)\). In addition, as the congruence
      is the trivial congruence when restricted to the idempotents,
      the trace of \(\rho\) is \(\Delta_{E(S)}\). Thus the kernel-trace method
      (see for example \cite[Theorem 5.3.3]{Howie}) tells us that \(\rho\) must be the trivial congruence.
      We have thus shown that every idempotent-pure congruence on \(S\) is trivial,
      and so \(S\) is \(E\)-disjunctive.
  \end{proof}

    We now consider various closure properties of the class of \(E\)-disjunctive
    inverse semigroups. This class is not closed under taking inverse
    subsemigroups. For example, every inverse semigroup is isomorphic to an
    inverse subsemigroup of some symmetric inverse monoid (by the Vagner-Preston Representation Theorem~\cite[Theorem 5.1.7]{Howie}), which is $E$-disjunctive (see \cref{ex:symmetric_inverse_monoid}).
    
    The class of $E$-disjunctive inverse semigroups is closed under passing to ideals.
    \begin{lem}
    \label{ideal_lem}
    Let \(I\) be an ideal of an \(E\)-disjunctive inverse semigroup \(S\). Then \(I\)
    is \(E\)-disjunctive.
  \end{lem}

  \begin{proof}
    Suppose that \(\rho\) is an idempotent-pure congruence on \(I\). Define the binary relation
    \(\bar{\rho}\) on \(S\) by
    \[
      \bar{\rho} = \rho \cup \Delta_S.
    \]
    We will show that \(\bar{\rho}\) is an idempotent-pure congruence. It is immediate that
    \(\bar{\rho}\) is an equivalence relation and idempotent-pure. We will show \(\bar{\rho}\)
    is a congruence. Let \(s, \ t, \ x \in S\) and suppose \(s \bar{\rho} t\). Then \(s = t\), in
    which case \(sx = tx\) and so \(sx \rho tx\), otherwise \(s \rho t\), and \(s, \ t \in I\).
    As \(I\) is an ideal, \(sx, \ tx \in I\) and so \(sx \rho tx\). Thus \(\bar{\rho}\) is an
    idempotent-pure congruence on \(S\). Since \(S\) is \(E\)-disjunctive, \(\bar{\rho}\) is trivial,
    and so \(\rho\) is trivial, and \(I\) is \(E\)-disjunctive.
  \end{proof}
  
In the other direction, \(E\)-disjunctivity is not closed under passing to
inverse supersemigroups, because adjoining an identity and then another identity will
result in a non-\(E\)-disjunctive semigroup. However, a semigroup will be
\(E\)-disjunctive if it has an \(E\)-disjunctive full inverse subsemigroup.
Recall that an inverse subsemigroup \(T\) of an inverse semigroup \(S\) is
\textit{full} if \(E(S) = E(T)\). We include the proof for completeness.

\begin{lem}[\cite{yoshida_E_dis}, Lemma 1]\label{full-subsemi}
    If \(S\) is an inverse semigroup with a full \(E\)-disjunctive subsemigroup \(T\), then \(S\) is \(E\)-disjunctive.
\end{lem}

\begin{proof}
    Any non-trivial idempotent-pure congruence on \(S\) identifies two idempotents of \(S\). Thus every congruence on \(S\) identifies two elements of \(T\). It follows that any non-trivial idempotent-pure congruence on \(S\), induces a non-trivial idempotent-pure congruence on \(T\), so none can exist.
\end{proof}

Another construction under which \(E\)-disjunctivity is preserved is taking finite direct products.

\begin{proposition}
    \label{prop:dir_prod}
    Let \(S_1\) and \(S_2\) be non-empty inverse semigroups. Then \(S_1\) and \(S_2\) are
    \(E\)-disjunctive if and only if \(S_1 \times S_2\) is \(E\)-disjunctive.
\end{proposition}

\begin{proof}
\((\Rightarrow)\)
    Note that \(E(S_1) \times E(S_2) = E(S_1 \times S_2)\).
    Let \((e_1, e_2), (f_1, f_2)\in E(S_1 \times S_2)\) be arbitrary.
    
    We show that \(\bar{\alpha} (e_1, e_2) \bar{\beta} \in E(S_1)\) if and only if
    \(\bar{\alpha} (f_1, f_2) \bar{\beta} \in E(S_1)\) for all \(\bar{\alpha},
    \bar{\beta} \in (S_1 \times S_2)^1\) implies that \((e_1, e_2)=(f_1, f_2)\).
    This shows that the trace of the syntactic congruence on \(S_1 \times S_2\)
    is just equality on the idempotents. The kernel of the syntactic congruence
    is \(E(S_1 \times S_2)\), and congruences of inverse semigroups are
    determined by their kernel and trace (see, for example Theorem 5.3.3 of
    \cite{Howie}), this implies that the syntactic congruence is equality (as it
    has the same kernel and trace).

    Suppose that the following statement holds for all \(\bar{\alpha},
    \bar{\beta} \in (S_1 \times S_2)^1\): \(\bar{\alpha} (e_1, e_2) \bar{\beta}
    \in E(S_1) \iff \bar{\alpha} (f_1, f_2) \bar{\beta} \in E(S_1)\). We show
    that \(e_1=f_1\), the other coordinate follows by symmetry. Let \(\alpha,
    \beta \in S_1^1\) be arbitrary.

By the previous statement, when \(\alpha, \beta \in S_1\) 
    \begin{align*}
        (\alpha, e_2f_2) (e_1, e_2) (\beta, e_2f_2) \in E(S_1\times S_2) &\iff (\alpha, e_2f_2) (f_1, f_2) (\beta, e_2f_2)\in E(S_1 \times S_2).
    \end{align*}
    In the case that \(\alpha\) or \(\beta\) is the identity \(1\) the above
    equivalence still holds as \((\alpha, e_2f_2)\) and/or (\(\beta, e_2f_2) \)
    can be replaced with \(1\in (S_1\times S_2)^1\) and the resulting statements
    are equivalent. So
       \begin{align*}
        (\alpha e_1\beta, e_2f_2) \in E(S_1\times S_2) &\iff (\alpha f_1\beta, e_2f_2) \in E(S_1\times S_2)
    \end{align*}
    and hence \(\alpha e_1\beta\in E(S_1)\) if and only if \(\alpha f_1\beta\). Since \(\alpha\)
    and \(\beta\) were arbitrary and \(S_1\) is E-disjunctive, it follows that
    \(e_1=f_1\).

    \((\Leftarrow)\) Suppose that \(S_1\times S_2\) is \(E\)-disjunctive. Let
    \(\rho\) be an idempotent-pure congruence on \(S_1\). We show that \(\rho\)
    is trivial. Let \(\rho'\) be the congruence on \(S_1\times S_2\) defined by
    \((s_1, s_2)\rho' (t_1, t_2)\) if and only if \(s_1 \rho t_1\) and
    \(s_2=t_2\). As \(\rho'\) is idempotent-pure, it follows that \(\rho'\) is
    equality. Hence if \(S_2\) is not empty, it follows that \(\rho\) is also
    equality.
\end{proof}

Finally, we will mention the following result in \cref{sec:mcealister}.

  \begin{proposition}[{\cite[Proposition 2.4.5]{inverse_semigroups}}]\label{prop:comp-rel}
      A congruence \(\rho\) on an inverse semigroup \(S\) is idempotent-pure if
      and only if \(\rho\) is contained in the compatibility relation $\set{(a,
      b)\in S^2}{ab^{-1}, a^{-1}b\in E(S)}$.
  \end{proposition}

\section{The natural partial order, identities and zeros}
\label{sec:npo10}

In this section, we consider the interaction of the notion of $E$-disjunctivity
and the natural partial order on any inverse semigroup, and some applications.
Recall that \textit{the natural partial order} $\leq$ on an inverse semigroup
\(S\) is defined by $s\leq t$ if there exists $e\in E(S)$ such that $s = et$.

We define another partial order \(\preceq\) on an \(E\)-disjunctive inverse
semigroup \(S\) so that \(s \preceq t\) if
\[
        \alpha s \beta \in E(S) \impliedby \alpha t \beta \in E(S)
\]
for all \(\alpha, \ \beta \in S^1\).

    \begin{proposition}
      \label{partial_order_prop}
      The partial order \(\preceq\) on an \(E\)-disjunctive inverse semigroup is
      equal to the natural partial order.
    \end{proposition}

    \begin{proof}
      Let \(S\) be an \(E\)-disjunctive inverse semigroup, and let \(s, \ t \in
      S\). Suppose \(s \leq t\). Then \(s = te\) for some \(e \in E(S)\).
      Let \(\alpha, \ \beta \in S^1\) be such that
     \(\alpha t \beta \in E(S)\). Then \(\alpha s \beta = \alpha te \beta \in
     E(S)\), and \(s \preceq t\).

     Suppose that \(s \preceq t\). Since \(t^{-1} t \in E(S)\), \(t^{-1} s \in
     E(S)\), and so \(s^{-1} t \in E(S)\). If \(\alpha, \ \beta \in E(S)\) are
     such that \(\alpha s \beta \in E(S)\), then \(\alpha ss^{-1}t \beta \in
     E(S)\), and so \(s s^{-1}t \preceq s\). If \(\alpha, \ \beta \in S\) are
     such that \(\alpha ss^{-1} t \beta \in E(S)\), then using the fact that
     \(s \preceq t\), it follows that \(\alpha s \beta = (\alpha ss^{-1}) s
     \beta \in E(S)\), and so \(s \preceq ss^{-1} t\). It follows that \(s\)
     is related to \(ss^{-1} t\) by the syntactic congruence. Since \(S\) is
     \(E\)-disjunctive, \(s = ss^{-1} t\), and so \(s \leq t\).
    \end{proof}

The next lemma provides a characterisation of the identity element of
an $E$-disjunctive inverse semigroup in terms of the natural partial order $\leq$.

\begin{lem}
  \label{EDis_1_lem}
  Let \(S\) be an \(E\)-disjunctive inverse semigroup, and let \(e \in S\) be such that
  \(\alpha e \beta \in E(S)\) if and only if \(\alpha \beta \in E(S^1)\) for all \(\alpha, \beta \in S^1\). Then \(e\) is an identity.
\end{lem}

\begin{proof}
  Using \cref{partial_order_prop}, \(e \geq x\) for all \(x \in E(S)\). It
  follows that \(xe = x = ex\) for all \(x \in E(S)\). If \(x \in S \setminus
  E(S)\), then \(xe = xx^{-1} x e = x (x^{-1}x e) = x (x^{-1}x) = x\). Similarly
  \(ex = exx^{-1}x = xx^{-1}x = x\). Thus \(e\) is an identity for \(S\).
\end{proof}

A corollary of the previous lemma characterises when an $E$-disjunctive inverse
semigroup $S$ with identity adjoined $S ^ 1$ is also $E$-disjunctive.

\begin{cor}
    \label{cor:adj-1}
    Let \(S\) be an \(E\)-disjunctive inverse semigroup. Then \(S^1\) is
    \(E\)-disjunctive if and only if \(S\) does not contain an identity.
\end{cor}

\begin{proof}
    \((\Rightarrow)\) We prove the contrapositive; that is that if \(S\)
    contains an identity, then \(S^1\) is not \(E\)-disjunctive. Suppose \(S\)
    contains an identity \(e\). Then for all \(x, y \in S^1\), \(xey \in
    E(S^1)\) if and only if \(xy \in E(S^1)\). Similarly, if \(1\) is the
    adjoined identity, \(x1y \in E(S^1)\) if and only if \(xy \in E(S^1)\) for
    all \(x, y \in S^1\). Thus \(1\) and \(e\) are related by the syntactic
    congruence in \(S^1\), and so \(S^1\) is not \(E\)-disjunctive.

    \((\Leftarrow)\) Suppose that \(S\) does not contain an identity. Since \(S\) is
    \(E\)-disjunctive, \cref{EDis_1_lem} tells us there is no element \(e \in
    S\) such that for all \(x, y \in S^1\), we have \(xey \in E(S^1)\) if and
    only if \(xy \in E(S)\). However, \(x1y \in E(S^1)\) if and only if \(xy \in
    E(S^1)\) for all \(x, y \in S^1\). It follows that \(1\) is not related to
    any other element of \(S^1\) by the syntactic congruence. As no two elements
    inside \(S\) are related by the syntactic congruence of \(S^1\), we have
    that the syntactic congruence of \(S^1\) is equality, and so \(S^1\) is
    \(E\)-disjunctive.
\end{proof}

The next lemma is an analogue of \cref{EDis_1_lem} where ``identity'' is replaced by ``zero''.

    \begin{lem}
      \label{EDis_0_lem}
      Let \(S\) be an \(E\)-disjunctive inverse semigroup, and let \(e \in S\) be such that
      \(\alpha e \beta \in E(S)\) for all \(\alpha, \ \beta \in S^1\). Then
      \(e\) is a zero.
    \end{lem}

\begin{proof}
  Using \cref{partial_order_prop}, \(e \leq x\) for all \(x \in
  S\). If \(x \in E(S)\), then \(xe \leq e\) and \(ex \leq e\). As \(e \leq xe\) and \(e \leq ex\), it follows that \(xe = e = ex\). If \(x \in S
  \setminus E(S)\), then \(xe, \ ex \in E(S)\). Thus \(xe = xe^2 = e\) and \(ex = e^2 x = e\), because we have shown that \(e\) acts as a zero when multiplied by elements of \(E(S)\). In particular, \(e\) is a
  (the) zero in \(S\).
\end{proof}

To obtain the analogue of \cref{cor:adj-1} we prove a more general result. 

\begin{proposition}
  \label{prop:0-direct}
  A \(0\)-direct union of \(E\)-disjunctive inverse semigroups \((S_i)_{i\in I}\) is \(E\)-disjunctive if and only if none of the semigroups \(S_i\) has a zero.
\end{proposition}
\begin{proof}
    \((\Rightarrow)\) We prove the contrapositive. Suppose that \(S_i\) has a
    zero element \(0_i\) for some \(i\in I\). Then the congruence on the zero
    direct union generated by the pair \((0_i, 0)\) identifies only these two
    elements and is thus non-trivial and idempotent-pure.

    \((\Leftarrow)\)
    Let \(\sigma\) be an idempotent-pure congruence on the zero direct union.
    Suppose for a contradiction that there is \((a, b)\in \sigma\) with \(a\neq
    b\). Since \(\sigma\) restricts to idempotent-pure congruences on each
    \(S_i\), it follows that \(\sigma\) is trivial on each \(S_i\). In
    particular, \(a\) and \(b\) do not belong to the same semigroup \(S_i\) for
    any \(i\in I\).

    Suppose without loss of generality that \((a, b)\in \sigma\), \(i\in I\),
    \(a\in S_i\) and \(b\not\in S_i\). Then \((a, 0)=(aa^{-1}a, ba^{-1}a) \in
    \sigma\). So \(a\) is the unique element of \(S_i\) related to \(0\) by
    \(\sigma\). Since \(\sigma\) is a congruence, it follows that the set
    \(\{a\}\) is an ideal of the semigroup \(S_i\). This is a contradiction as
    \(S_i\) does not contain a zero. 
\end{proof}

\begin{cor}\label{cor:add_zero}
  Let \(S\) be an \(E\)-disjunctive inverse semigroup. Then \(S^0\) is
  \(E\)-disjunctive if and only if \(S\) does not contain a zero.
\end{cor}

\section{Wreath products and quotients}
\label{sec:wreath-prod}

In this section we consider when wreath products of inverse semigroups are
$E$-disjunctive and use this to show that every inverse semigroup is a
homomorphic image of an \(E\)-disjunctive inverse semigroup. We think of wreath
products in terms of matrices. Recall that an element of a wreath product of
groups \(G\wr_X H\) where \(H\leq S_X\) consists of a pair \(((g_x)_{x\in
X},h)\in (G^X, H)\). We think of the elements of \(G\wr_X H\) as an \(X\times
X\) matrix \(M\) such that the entry indexed by $(x, y) \in X \times X$ in $M$
is \(g_x\) whenever \((x)h = y\) and $0$ otherwise. We will also think of such
matrices as functions $M:X \times X\to G\cup \{0\}$, where $M(x, y)$ is just the
$(x, y)$-entry of the matrix. It is routine to verify that the group \(G\wr_X
H\) is isomorphic to the group consisting of the corresponding matrices, as just
defined, where $0 + g = g + 0 = g$ for all $g\in G$.

We extend this definition of wreath products to inverse semigroups $S$ and
subsemigroups $T$ of the symmetric inverse monoid $I_X$ as follows. To do this
nicely, we introduce a semiring which contains \(S\) but only use the
\(S\cup\{0\}\) part of it.

Recall that \((R, +, \cdot)\) is a \textit{semiring} if the following hold:
\begin{enumerate}
    \item \((R, +)\) is a commutative monoid whose identity we call \(0\);
    \item \((R, \cdot)\) is a semigroup;
    \item the operation $\cdot$ distributes over \(+\); and 
    \item \(r\cdot 0=0\cdot r= 0\) for all \(r\in R\).
\end{enumerate}

If \(S\) is a semigroup, then we define \(\mathbb {N}[S]\) to be the quotient of
the free semiring over the set \(S\) by the relations \(s\cdot t=st\) \(s, t\in
S\). That is, $\mathbb{N}[S]$ consists of finite formal sums of the form
\[
\sum_{s\in S} n_s s
\]
where $n_s\in \N$ for all $s\in S$ and only finitely many $n_s$ are non-zero
with the natural multiplication. 

If \(R\) is a semiring and \(X\) is a set, then we define the \textit{(row
finite) matrix semiring} over \(R\) by
\begin{align*}
    M_{X}(R) & = \set{f\colon X \times X \to R }{\text{all but finitely many entries of each row of }f\text{ are } 0}\\
    & = \set{f\colon X \times X \to R}{\text{for all }x \in X,\ |(R\setminus\{0\})f^{-1} \cap (\{x\}\times X)|<\infty}
\end{align*}
with operation defined by
    \[(x, y)fg= \sum_{i\in X} (x, i)f\cdot (i, y)g\]
for \(f, g\in M_{X}(R)\).

If \(S\) is a semigroup, \(P_X\) is the partial transformation monoid on the set
\(X\), and \(T\leq P_X\), then we define \(S\wr T\) to be the following submonoid
of \(M_X(\N[S])\):
\[S\wr T=\set{f\in M_X(\N[S])}{\im(f)\subseteq S \cup \{0\}\text{ and }(S)f^{-1} \in T}.\]
The condition $(S)f^{-1}\in T$ makes sense since $(S)f^{-1}\subseteq X\times X$,
and so this condition simply asserts that the relation $(S)f^{-1}$ is a partial
transformation that belongs to $T$.

We define \(\phi \colon S\wr T \to T\) by \((f)\phi= (S)f^{-1}\). 
It is routine to verify that $\phi$ is a surjective homomorphism. As such 
the multiplication in $S\wr T$ can alternatively be defined as follows 
\[(x, y)fg= 
\begin{cases}
    (x, z)f\cdot (z, y)g &  \text{ if } (x, z) \in (f)\phi,\ (z, y) \in (g)\phi\\
    0 & \text{ if }(x, y) \notin (fg)\phi,
\end{cases}
\]
(since the sum only ever has one non-zero summand).

\begin{lem}
  If \(S\) and \(T\leq I_X\leq P_X\) are inverse semigroups, then \(S \wr T\) is an inverse semigroup.
\end{lem}

\begin{proof}
If \(f \in S \wr T\), then it is routine to verify that $f$ is an idempotent if and only if 
\begin{enumerate}
    \item the preimage of \(S\) under \(f\) is an idempotent of \(T\);
    \item the image of \(f\) contains only idempotents.
\end{enumerate}
Since \(S\) and \(T\) are inverse semigroups, \(f\in S\wr T\) is a diagonal
matrix and so the idempotents of \(S\wr T\) commute. Also if $f\in S\wr T$, then we define
$f ^ {-1}$ to be
\[(a, b)f^{-1}= \begin{cases}
    ((b, a)f) ^{-1} &  \text{ if }(b, a)f \neq 0\\
    0 & \text{ if }(b, a)f = 0
\end{cases}.
\]
In other words, $f^ {-1}$ is obtained from $f$ by transposing $f$ and inverting
its entries. It is straightforward to show that $f^{-1}$ is a semigroup
theoretic inverse of $f$, and so \(S\wr T\) is an inverse semigroup.
\end{proof}

The following proposition is a special case of \cref{thm_final_wreath}, however we include the proof below, as
it is more straightforward, and helps exhibit the ideas behind the proof of \cref{thm_final_wreath}.

\begin{proposition}
    \label{prop:first_wreath}
    Let \(G\) be a non-trivial group and \(T\leq I_X\) be an inverse semigroup. 
    Then \(G \wr T\) is \(E\)-disjunctive.
\end{proposition}
\begin{proof} Seeking a contradiction suppose that \(\rho\) is a non-trivial
   idempotent-pure congruence on \(G \wr T\). Then by \cref{lem-idempotents}
   there exist \(f\neq g\in E(G \wr T)\) such that \((f, g)\in \rho\). If \(fg =
   f\), then \(f < g\). If \(fg \neq f\), then \(f < fg\) are idempotents such
   that \((f, fg) \in \rho\). Thus we may assume without loss of generality that
   \(f<g\). Note that \(f, g\) are both matrices whose entries are all \(0\)
   except for some idempotents on the diagonal. Since $f < g$, there exists
   \(x\in X\) such that \((x, x)f<(x, x)g\). The entries of $f$ and $g$ belong
   to $G\cup \{0\}$ (whose only idempotents are the identity $1_G$ and $0$).
   Thus
    \((x, x)g = 1_G\) and \((x, x)f = 0\).
   Let \(h\in G\setminus \{1_G\}\) and let \(g'\in S \wr T\) be the matrix with
   the same entries as \(g\) except that \((x, x)g'=h\). Thus \(g'g = g'\) and
   \(g'f = f\). It follows that
   \[(f, g)\in \rho \Rightarrow (g'f, g'g)\in \rho \Rightarrow (f, g')\in \rho .\]
   But \(g'\) is not an idempotent, and \(\rho\) is idempotent-pure. This is a contradiction.
\end{proof}

The next theorem establishes that every inverse semigroup is a quotient of some $E$-disjunctive semigroup.

\begin{theorem}\label{thm_final_wreath}
    Let \(S\) be an \(E\)-disjunctive inverse semigroup without a zero and \(T\leq I_X\) be a inverse semigroup.
    Then \(S \wr T\) is \(E\)-disjunctive and has \(T\) as a quotient.

    Moreover, if \(T\leq S_X\), then the assumption that \(S\) has no zero can be dropped.
\end{theorem}
 \begin{proof}
    If \(T\) is empty, then \(S \wr T\) is empty, and so \(S \wr T\) is
    \(E\)-disjunctive, as required. We therefore assume that \(T\) is non-empty.
 
   Let \(D\) be equal to the set of diagonal matrices of \(S \wr T\). Then \(D\) is a full inverse subsemigroup of \(S \wr T\), i.e. $D$ contains all of the idempotents of $S\wr T$.
   By \cref{full-subsemi}, it suffices to show that \(D\) is \(E\)-disjunctive.

   Let \(\rho\) be an idempotent-pure congruence on \(D\). We show that \(\rho\)
   is trivial. It suffices by \cref{lem-idempotents} to show that \(\rho\) is
   trivial on the idempotents. Suppose that \((f, g)\in \rho\) and \(f,g\) are
   idempotents. Let \(h=fg \leq f\). We show that \(f=h\), then by symmetry we
   will have \(g=h\) and hence
   \(f=g\) as required.

    Let \(x\in X\) be arbitrary. 
    We need only show that \((x,x)f=(x,x)h\) (as they are idempotents they agree on their non-diagonal entries). 
    There are two cases to consider.
    \medskip

   \noindent \textbf{Case 1:} \((x, x)f\geq(x, x)h =0\).
   We define
   \[I=\makeset{s\in S}{the matrix obtained from \(f\) by replacing \((x, x)f\) with \(s\) belongs to \( h/\rho\) }.\]
   For all \(s\in S\cup\{0\}\), let \(f_s\) be the matrix obtained from \(f\) by replacing \((x, x)f\) with \(s\). It follows that 
   \[f_S\coloneqq \makeset{f_s}{\(s\in S\cup \{0\}\)}\]
   is a semigroup isomorphic to \(S\cup \{0\}\).
   The restriction of \(h/\rho\) to $f_S$ is \(I\).
   The natural map from \(S\) to \(f_S\) embeds \(I\) into a congruence class of \(f_S\) containing the zero element of $f_S$.
   Thus \(I\) is an ideal of \(S\) containing \((x, x)f\).
   Since \(\rho\) is idempotent-pure, it follows that \(I\) consists of idempotents.
   Since \(S\) is \(E\)-disjunctive, it follows that \(I\) must therefore be a singleton (otherwise \(S/I\) would be a proper idempotent-pure quotient).
   Hence, since \(I\) is a singleton ideal, the unique element of \(I\) is a zero for \(S\), a contradiction. So in fact Case 1 never occurs.
   \medskip

 Showing that Case 1 does not occur is the only point in the proof requiring that \(S\) has no zero element. When \(T\leq S_X\), this case does not occur because the only idempotent of \(T\) is the identity function on \(X\), so \(h\) has no zeros on the diagonal. Hence why the assumption is no longer needed.
   \medskip
   
   \noindent \textbf{Case 2:} \((x, x)f\geq (x, x)h >0\).
   For all \(s \in S\) let \(h_s\in D\) be the element which agrees with \(h\) on all entries except \((x, x)h_s=s\).
   Then \(S_x\coloneqq \makeset{h_s}{\(s\in S\)}\) is a subsemigroup of \(D\) isomorphic to \(S\).
   Hence, since \(\rho\) is idempotent-pure, the restriction of \(\rho\) to \(S_x\) is trivial.
   In particular, to show that \((x,x)h=(x, x)f\), we need only show that \((h, h_{(x, x)f} = (h_{(x,x)h}, h_{(x,x)f})\in \rho\).

   We denote  $h_{(x,x)f}$ by \(f'\). 
   Since \(h\leq f\), \(f'f=f'\) and \(f'h=h\). It follows that
   \begin{align*}
     (f, g)\in \rho&\Rightarrow  (ff, fg)\in \rho\\
     &\Rightarrow  (f, h)\in \rho\\
      &\Rightarrow  (f'f, f'h)\in \rho\\
       &\Rightarrow  (f', h)\in \rho
   \end{align*}
   as required.
\end{proof}

\begin{cor}
    \label{cor:all-qts}
    Every inverse semigroup is a quotient of an \(E\)-disjunctive inverse semigroup.
\end{cor}

\part{Some classes of $E$-disjunctive inverse semigroups}

In this part of the paper we provide a number of examples of $E$-disjunctive inverse semigroups \cref{section-compendium}, and we characterise $E$-disjunctive inverse semigroups belonging to the classes of: graph inverse semigroups \cref{section-gis} (in terms of the underlying graphs); and monogenic inverse semigroups \cref{section-monogenic}. As mentioned above, the Clifford $E$-disjunctive semigroup were characterised in~\cite[Theorem 6]{Clifford_IP}. 

\section{A compendium of examples}\label{section-compendium}
In this section we give various examples of \(E\)-disjunctive inverse semigroups. 
These serve as counterexamples to various natural questions about \(E\)-disjunctive inverse semigroups.

Recall that if $X$ is any set, then the \textit{symmetric inverse monoid} $I_X$ on $X$ (often written \(I_n\), if \(|X| = n\) is finite) consists of the bijections between subsets of $X$ and the operation $\circ$ is the usual composition of binary relations. That is, if $f, g\in I_X$, then 
\[
f\circ g = \{(x, z)\in X \times X \mid \text{there exists }y\in X \text{ such that }
(x, y)\in f \text{ and } (y, z)\in g\}.
\]
We may sometimes, arbitrarily, write $fg$, or $f\cdot g$, instead of $f\circ g$.

By the Vagner-Preston Theorem~\cite[Theorem 5.1.7]{Howie}, every inverse semigroup is isomorphic to an inverse subsemigroup of some symmetric inverse monoid.

\begin{ex}
  \label{ex:symmetric_inverse_monoid}
  The symmetric inverse monoid \(I_X\) on a set \(X\) is
  \(E\)-disjunctive if and only if \(|X| \neq 1\).
\end{ex}
\begin{proof}
If $|X|=0$, then $I_X$ is the trivial semigroup, and hence is $E$-disjunctive. 

If $|X| =1$, then $I_X$ is a semilattice of size $2$, and as such is not $E$-disjunctive since it is non-trivial and \(E\)-unitary, and hence the minimum group congruence is
a non-trivial idempotent-pure congruence.

  Suppose that $|X|\geq 2$. Let \(s,  t \in I_X\) and suppose that  $(s, t)$ belongs to the syntactic congruence on $I_X$. Then \(\alpha s \beta
  \in E(I_X)\) if and only if \(\alpha t \beta \in E(I_X)\) for all $\alpha, \beta\in I_X$. If $s\in I_X$, then we consider $s$ as the subset of $X \times X$ consisting of the pairs $(x, (x)s)$.
  Let \(x, y \in X\), and  let \(z \in X \setminus \{x\}\).
  Then \(\{(x,x)\} \circ s \circ \{(y, z)\} \notin E(I_X)\) if and only
  if \((x, y) \in s\); and similarly for $t$. Since \(\{(x, x)\} \circ s \circ \{(y, z)\} \in E(I_X)\)
  if and only if \(\{(x, x)\} \circ t \circ \{(y, z)\} \in E(I_X)\), and so
 \((x, y) \in s\) if and only if $(x, y) \in t$. Thus \(s = t\), and so the syntactic
  congruence on $I_X$ is equality.
\end{proof}

A congruence \(\rho\) is \textit{idempotent-separating}
if \(\rho\) never relates two distinct idempotents. Thus
idempotent-separating is a dual property to idempotent-pure
and inverse semigroups with no non-trivial idempotent-separating congruences, called \textit{fundamental} inverse
semigroups are a dual class of inverse semigroups to
\(E\)-disjunctive inverse semigroups. It is not difficult
to show that the symmetric inverse monoid is fundamental,
thus providing a non-congruence free example of a semigroup
in both of these classes.

Recall, from \cite{DualSym}, for example, that the \textit{dual symmetric inverse monoid} \(I_X^*\) 
is defined as follows.
The underlying set of \(I_X^*\) is the set of partitions of \(X \times
\{0, 1\}\) such that each part intersects both \(X \times
\{0\}\) and \(X \times
\{1\}\). In other words, elements of $I_X^*$ correspond to bijections between the parts of a partition of $X\times \{0\}$ and a partition of $X\times \{1\}$.
We will use partitions and the corresponding equivalence relations
interchangeably. 
Given $s, t \in I_X^*$, we define \(\operatorname{Diag}(s,t)\) to be the least equivalence relation on
\(X \times \{0, 1, 2\}\) containing
\[\{((x, a), (y, b))\in (X\times \{0,1,2\})^2:((x, a), (y, b))\in s \text{ or }
  ((x, a-1), (y, b-1))\in t\}.\]
The product of \(s\) and \(t\) is defined to be
\[\{((x,a),(y,b))\in (X\times \{0, 1\})^2: ((x,2a),(y,2b))\in
  \operatorname{Diag}(s,t)\}\]
and is denoted $st$. 
Note that $e\in I_X^*$ is an idempotent whenever $((x, 0), (y, 0))\in e$ if and only if $((x, 1), (y, 1))\in e$ for all $x, y \in X$ (i.e. the partitions of $X\times \{0\}$ and $X\times \{1\}$ are the ``same'' and the corresponding function is the identity).

\begin{ex}
  \label{ex:dual_symmetric_inverse_monoid}
  The dual symmetric inverse monoid $I_X^*$, where $X$ is any set, is $E$-disjunctive. 
\begin{proof}
    If $I_X^*$ has at most one idempotent, then  \(I_X^*\) is either a group or the empty semigroup. In either case, $I_X^*$ is \(E\)-disjunctive. By \cref{lem-idempotents}, it suffices to show that every idempotent-pure congruence is trivial on \(E(S)\).
    Let \(\rho\) be an idempotent-pure congruence on \(I_X ^ *\). Suppose
    for contradiction that there exist distinct idempotents \(e, f \in E(I_X ^*)\) such that \((e, f)\in \rho\). 
    Since \(e \neq f\), at most one of \(e\) and \(f\) equals \(ef\); assume without loss of generality that \(e\neq ef\).
    Then there is a part of \(ef\) which is a union of at least two parts of \(e\). Let \(s\in I_X^*\) be an element which swaps two of these parts of \(e\) and fixes the others. If \((e, f) \in \rho\), then $(e, ef) = (e ^ 2, ef) \in \rho$ and so $(se, sef) \in \rho$.
    But $se = s$, as \(e\) acts as the identity function on the image of \(s\), which is not an idempotent. On the other hand, \(sef = ef \in E(I_X^*)\), and so $(s, ef)\in \rho$. Since $ef$ is an idempotent and $s$ is not, this contradicts
    \(\rho\) being idempotent-pure, and so \(I_X^*\) is \(E\)-disjunctive.
\end{proof}
\end{ex}

The next example shows that ``congruence'' cannot be replaced by ``right congruence'' in the definition of $E$-disjunctive inverse semigroups. In fact, the next example is the unique inverse semigroup of smallest size up to isomorphism, showing this. 

\begin{ex}
\label{ex:GAP}
Let $S$ be the inverse semigroup defined by the inverse semigroup presentation:
\[
\langle x,\ y \mid
  xy= xy^{-1}= yx=  yx^{-1}= y^{-1}x^{-1} = x ^2,\ yy^{-1} = xx^{-1}
\rangle.
\]
It can be shown, for example, using GAP~\cite{GAP4, Mitchell2023aa}, that: 
\[S = \{x, y, x^{-1}, y^{-1}, x^2, xx^{-1}, x^{-1}x, x^{-1}y, y^{-1}x, y^{-1}y, x^3\},\]
that $S$ is $E$-disjunctive, 
the least right congruence $\rho$ on $S$ containing the pair $(x ^ 3, x)$ has non-trivial classes:
\[
\{ x, y, x^3 \},\  \{x^2, xx^{-1}\},
\]
and that the idempotents of $S$ are:
\[
x^{-1}x,\ y^{-1}y,\ xx^{-1},\ x^2.
\]
Hence $\rho$ is idempotent-pure, using the obvious definition of this notion for right congruences. It is also possible to show using~\cite{MalandroNonLatticesPage, MalandroLatticesPage}, based on \cite{Malandro}, that there is no smaller $E$-disjunctive inverse semigroup admitting such a right congruence, that $S$ is the unique inverse semigroup (up to isomorphism) of size $11$ admitting such a right congruence, and even that $\rho$ is the only such right congruence on $S$.

      \begin{figure}
      \begin{tikzpicture}
        [scale=.5, auto=left,every node/.style={rectangle}]
        
        \node (main) at (0, 5) {\(
        \arraycolsep=3pt\def\arraystretch{1.5}
        \begin{array}{|c|c|c|}
            \hline
            xx^{-1} & x^{-1} & y^{-1} \\ \hline
           x  & x^{-1}x & y^{-1}x \\ \hline
           y  & x^{-1}y & y^{-1}y \\ \hline
        \end{array}\)};
        \node (C2) at (0, 0) {\(\begin{array}{|c|}
            \hline x^2, x^3 \\ \hline 
        \end{array}\)};
        \draw (main) to (C2);

      \end{tikzpicture}
      \caption{Egg-box diagram of an \(E\)-disjunctive inverse semigroup with a non-trivial
      idempotent-pure right congruence.}
      \label{idem_pure_rt_congs_fig}
    \end{figure}
\end{ex}

For our next example, we require the definition of Thompson's group \(V\), which we define below. For
a more comprehensive introduction to Thompson's group $V$, we refer the reader to \cite{CannonFloydParry}.

\begin{dfn}[Thompson's group $V$.]
      Let \(\mathfrak{C}\) denote the Cantor space, with underlying set \(\{0, \ 1\}^\omega\) (that is,
  infinite sequences of \(0\)s and \(1\)s), and using the product topology induced from the discrete
  topology of \(\{0, \ 1\}\). 
  We denote the free monoid on $\{0, 1\}$ by $\{0, 1\} ^ *$ (i.e. the monoid of finite sequences of $0$s and $1$s with concatenation as the operation).
  If \(w \in \{0, \ 1\}^\ast\), then we define
  \[
  w \mathfrak{C} = \makeset{x\in  \mathfrak{C}}{\(w \text{ is a prefix of } x\)}.
  \]
Note that these sets are clopen, and the collection of all such sets is a basis for \(\mathfrak{C}\).

Let \(F_1\) and \(F_2\) be finite subsets of \(\{0, \ 1\}^\ast\) such that $|F_1| = |F_2|$, and 
\[
\makeset{w \mathfrak{C}}{\(w\in F_1\)} \quad  \text{ and } \quad \makeset{w\mathfrak{C}}{\(w\in F_2\)}\]
are partitions of \(\mathfrak{C}\). We call such subsets of \(\{0, 1\}^\ast\) \textit{complete antichains}. If \(u \in \mathfrak C\), then 
since $F_1$ partitions $\mathfrak{C}$ there exists $w_u\in F_1$ that is a prefix of $u$. 
In this case, we write $u = w_uv_u$ where $v_u\in \mathfrak{C}$  is just the suffix of $u$ following $w_u$.
The \textit{prefix exchange map} \(f \colon ~\mathfrak{C}~\to~\mathfrak{C}\) \textit{between} \(F_1\) and \(F_2\) induced by a bijection \(\phi \colon F_1 \to F_2\) is defined by
\[
(u)f = (w_u \phi) v_u.
\]
Every such prefix exchange map is a homeomorphism of $\mathfrak{C}$. The group of all prefix exchange maps between any pair of complete antichains is
called \textit{Thompson's group \(V\)}.
\end{dfn}
The following example is a slight modification of the Thompson inverse monoid \(\operatorname{Inv}_{2,1}\) introduced in \cite{birget2016monoid} (we modify it as that monoid does not have infinitely many \(\mathscr{J}\)-classes).

\begin{theorem}
      \label{ex:thompsonsV}
There exists a finitely generated \(E\)-disjunctive inverse monoid with infinitely many $\mathscr{J}$-classes.
\end{theorem}
\begin{proof}
  We give an example of a Thompson's group-like inverse monoid that is finitely generated and has infinitely
  many \(\mathscr{J}\)-classes.

  Let \(M\) be the inverse submonoid of the inverse monoid of partial permutations on \(\mathfrak{C}\) generated by
    Thompson's group \(V\) and
  the identity functions on \(1 \mathfrak{C}\) and
      $\{1^n0^\omega :n\in \N \cup \{0\}\}$. We denote the second of these identity functions by $e$.
  As Thompson's group \(V\) is $2$-generated (see for example \cite{BleakQuick}), 
  \(M\) is $4$-generated. 

  We next show that \(M\) has infinitely many \(\mathscr{J}\)-classes. We do so by showing that $M$ contains the identity function on a set of size $n$ for all \(n\in \N \cup \{0\}\). For different \(n\), these elements are not \(\mathscr{J}\)-related in \(I_{\mathfrak{C}}\) and hence \(M\) has infinitely many \(\mathscr{J}\)-classes.
 It suffices to show that every identity $f_n$ on the set 
 \[\{1^i0^\omega : i < n\}\]
 belongs to $M$. We denote 
  the identity function on the set \(\bigcup_{i< n}1^i0\mathfrak{C}\) by $g_n$. 
  It is straightforward to verify that $f_n = g_n e$, and so it suffices to show that $g_n\in M$.
 Note that
  \[
  1\mathfrak{C}= 1^n\mathfrak{C} \cup \bigcup_{1\leq i<n}1^i0\mathfrak{C}
  \]
  for any \(n\geq 1\). 
  If $F$ is any complete antichain in $\{0, 1\} ^ *$ containing $0$ and $1^n$, and $\phi\colon F \to F$ is the bijection swapping $0$ and $1 ^ n$. Then the corresponding prefix exchange map $\psi\in V$ maps $1\mathfrak{C}$ to $\dom(g_n)$. Hence 
   the conjugate of the identity function on \(1\mathfrak{C}\) by $\psi$ is \(f_n\). In particular, since the identity on \(1\mathfrak{C}\) is a generator of $M$,  $f_n$ belongs to $M$.
  
 We now show that \(M\) is \(E\)-disjunctive. 
 If $m\in M$ is arbitrary and the domain of $m$ is uncountable, then $m$ is a product of elements of $V$ and the second generator. Since the domain of the second generator is clopen and the elements of $V$ are prefix exchange maps, it follows that  the image of $m$ is clopen. Hence 
 every element \(m\) of \(M\) has a domain which is either clopen or countable.

If $J$ is the $\mathscr{J}$-class of $f_1$, then $J = Vf_1V$ and so $J$ consists of 
 functions with domain of size \(1\) which map an element with an infinite tail of zeros to another such element.

\begin{claim} If \(a,b\in M\setminus \{\varnothing\}\) are distinct, there is an idempotent \(e\in J\) such that \(ea, eb\in J \cup \{\varnothing\}\) and $ea\neq eb$.
\end{claim}
\begin{proof}
Suppose that $a$ and $b$ have the same domain. 
If \(\dom(a)= \dom(b)\) is countable, then there is an element \(a'\in J\) with $a'$ less than \(a\) in the usual partial order of inverse semigroups. 
Since $a\neq b$, there exists $u\in \mathfrak{C}$ such that $(u)a\neq (u)b$. In particular, we may choose $e\in J$  such that $e$ is the identity on $u$. In this case, $(u)ea= (u)a\neq (u)b= (u)eb$ and $ea, eb\in J$.

If \(\dom(a)= \dom(b)\) is clopen, then the set \(\makeset{x\in \dom(a)}{\((x)a \neq (x)b\)}\) is open and non-empty. Thus it contains a set \(w\mathfrak{C}\) for some $w\in \{0, 1\}^*$.  
By prefix replacement via Thompson's group $V$ (using the prefix \(w\)) we can find \(m'\in J\) such that \(\im(m')\subseteq w\mathfrak{C}\). Hence \({m'}^{-1}m'a, {m'}^{-1}m'b\in J\) are distinct and \({m'}^{-1}m'\) is the required idempotent in this case. 

If \(\dom(a)\neq \dom(b)\), then suppose without loss of generality that there is some \(u \in \dom(a) \setminus \dom(b)\) such that \(u\) ends with an infinite tail of zeros.
If we set $e$ to be the identity function on \(\{u\}\), then $e\in J$ since \(J\) comprises all functions with a domain of size \(1\) which map an element with an infinite tail of zeros to another such element.
\end{proof}

As \(J\) contains both idempotents and non-idempotents, it suffices to show that every non-trivial congruence on \(M\) identifies the \(\mathscr{J}\)-class $J$ with the 
$\mathscr{J}$-class of $f_0 = \varnothing$.

Let \(\rho\) be a non-trivial congruence on \(M\). Let \(a, b\in M\) be such that \(a\neq b\) and \((a, b)\in \rho\). By the claim above, there is \(e\in J\) such that \(ea, eb\in J\) and \(ea\neq eb\). Thus \(ea(ea)^{-1}\) and \(eb(ea)^{-1}\) are related by \(\rho\). But one of these is zero and the other is not, so all elements of \(J\) are related to zero by \(\rho\) and we are done.
\end{proof}
  The \textit{arithmetic inverse monoid} \(\mathcal A\), from \cite{Hines}, is the
  inverse submonoid
  of the symmetric inverse monoid \(I_{(\mathbb{Z}_{\geq 0})}\) generated by the set \(\{R_{a, b} \mid a, \ b \in
  \mathbb{Z}_{\geq 0}, \ a > b\}\), where \(R_{a, b} \in I_{(\mathbb{Z}_{\geq 0})}\) is
  defined by
  \[
    (n) R_{a, b} = 
    \begin{cases}
      \frac{n - b}{a} & n \equiv b \mod a \\
      \text{undefined} & \text{otherwise.}
    \end{cases}
  \]
 By \cite[Theorem 12]{Hines}, every non-zero element of \(\mathcal{A}\) may be written uniquely in the form
  \(R_{a, b} R_{c, d}^{-1}\), where \(c > d\) and \(a > b\) (note \((n)R_{c, d}^{-1} = nc + d\)). Since idempotents of
  \(I_{(\mathbb{Z}_{\geq 0})}\) are partial identities, it follows that the idempotents of
  \(\mathcal{A}\) are just \(R_{a, b} R_{a, b}^{-1}\) or \(\varnothing\) (where \(\varnothing\)
  is the empty map), for all \(a > b\) or the
  empty map \(\varnothing\). By \cite[Theorem 24]{Hines},
  if \((R_{a, b} R_{c, d}^{-1}), \ (R_{e, f} R_{g, h}^{-1}) \in \mathcal{A}
  \setminus \{\varnothing\}\), then
  \begin{equation}\label{eq-arith-prod}
    (R_{a, b} R_{c, d}^{-1}) \cdot (R_{e, f} R_{g, h}^{-1}) = \left\{
    \begin{array}{cl}
    R_{\frac{ae}{\gcd(c, e)}, \frac{a(r - d)}{c} + b} R_{\frac{gc}{\gcd(c, e)}, \frac{g(r - f)}{e} + h}^{-1} & \text{if }\gcd(c, \ e) \text{ divides } (d - f) \\
    \varnothing & \text{otherwise}, 
    \end{array}
    \right.
  \end{equation}
  where \(r\) is minimal such that \(r \equiv d \mod c\)
  and \(r \equiv f \mod e\). 
\begin{theorem}
  \label{ex:AIM}
    The arithmetic inverse monoid $\mathcal{A}$ is \(E\)-disjunctive.
\end{theorem}
\begin{proof}
  We will show that $\mathcal{A}$ is $E$-disjunctive by showing that the syntactic congruence is trivial. 
  Recall that the right syntactic congruence of a semigroup $S$ is \(\{(s, t) \in S \mid sx \in E(S) \iff tx \in
  E(S) \text{ for all } x \in S\}\). The right syntactic congruence is the maximum idempotent-pure right congruence on $S$.
  Since the syntactic congruence is idempotent-pure and a right congruence, the syntactic congruence is contained in the right syntactic congruence.
  Since the syntactic congruence $\rho$
  is idempotent-pure, the kernel of \(\rho\) is \(E(\mathcal{A})\), and so \(\rho\) is trivial if and only if $\rho$ has a
  trivial trace. Hence it suffices, since the trace of \(\rho\) is contained in the trace of the right syntactic congruence, to show that the trace of the syntactic right congruence is trivial. By definition, the right syntactic congruence has a trivial trace if and only if each idempotent \(e\) of \(\mathcal{A}\) defines a unique set \(Y_e = \{s \in \mathcal{A} \mid es \in E(\mathcal{A})\}\).
  
  Let \(R_{a, b} R_{a, b}^{-1} \in E(\mathcal{A})
  \setminus \{\varnothing\}\). We will describe the set $Y_{R_{a, b} R_{a, b}^{-1}}$.
  Suppose that \(R_{e, f} R_{g, h}^{-1} \in \mathcal{A} \setminus\{\varnothing\}\) is such that \((R_{a, b} R_{a, b}^{-1}) \cdot (R_{e, f} R_{g, h}^{-1}) \in E(\mathcal{A})\). Then $(R_{a, b} R_{a, b}^{-1}) \cdot (R_{e, f} R_{g, h}^{-1}) = \varnothing$ or $=R_{x,y}R_{x, y}^{-1}$ for some $x$ and $y$. Hence by \eqref{eq-arith-prod}
  precisely one of
  the following holds:
  \begin{enumerate}
      \item  \(ae = ga\) and \(r = \frac{g(r - f)}{e} + h\), where \(r\) is minimal such that
  \(r \equiv b \mod a\) and \(r \equiv f \mod e\);
      \item \(\gcd(a, \ e)\) does not divide \(b - f\).
  \end{enumerate}
  If (1) holds, then \(a e = ga\) implies \(e = g\), and so 
  \[
    r = \frac{g(r - f)}{e} + h = r - f + h.
  \]
  Hence \(f = h\). So \(R_{e, f} R_{g, h}^{-1} = R_{e, f} R_{e, f}^{-1} \in E(\mathcal{A})\).
If (2) holds, then immediately from \eqref{eq-arith-prod} the product \((R_{a, b} R_{a, b}^{-1}) \cdot (R_{e, f} R_{g, h}^{-1}) = \varnothing$.
  Let
  \[
    X_{a, b} = \{(e, \ f) \in \mathbb{Z}_{\geq 0} \times \mathbb{Z}_{\geq 0} \mid e > f \text{ and }
    \gcd(a, \ e) \text{ does not divide } b - f\}.
  \]
  Suppose \(R_{a', b'} R_{a', b'}^{-1} \in E(\mathcal{A})\) is such that \(X_{a', b'} =
  X_{a, b}\). Then for all \(e > f \in \mathbb{Z}_{\geq 0}\), we have that
  \(\gcd(a, \ e)\) divides \(b - f\) if and only if \(\gcd(a', \ e)\) divides \(b' - f\).
  We will show that \(a = a'\) and \(b = b'\).
  
  Let \(p\) be a prime greater than \(a\) and \(a'\) (and thus coprime to both). Then
  \(\gcd(a, \ pa')\) divides
  \(0 = b - b\), and so \(a' = \gcd(a', \ pa')\) divides \(b' - b\). By symmetry, 
  \(a\) divides \(b' - b\). If $f < pa'$, then  \(\gcd(a, \ a') = \gcd(a, \ pa')\) divides
  \(b - (b' + f)\) if and only if \(\gcd(a', \ pa')\) divides \(f\). Since
  \(b - b'\) is a multiple of \(a'\), and hence a multiple of \(\gcd(a, \ a')\) also, \(\gcd(a, \ a')\)
  divides \(b - (b' + f)\) if and only if \(\gcd(a, \ a')\) divides \(f\). Thus for all \(f < pa'\), \(\gcd(a, \ a')\) divides \(f\) if and
  only if \(a' = \gcd(a', \ a')\) divides \(f\). In particular, \(\gcd(a, \ a')\) divides
  \(a\), and so \(a'\) divides \(a\). By symmetry, \(a\) divides \(a'\), and since
  \(a, \ a' \in \mathbb{Z}_{> 0}\), it follows that \(a = a'\). It remains to show that \(b = b'\). We have already shown that \(a\) divides \(b' - b\).
  Since \(b' < a\) and \(b < a\), \(|b' - b| < a\) and so \(|b' - b| = 0\), as
  required.

  It follows that if \(R_{a, b} R_{a, b}^{-1} \in E(\mathcal A) \setminus \{\varnothing\}\), then
  \[
  Y_{R_{a, b}R_{a, b}^{-1}}\setminus E(\mathcal{A}) = \{R_{e, f} R_{g, h}^{-1} \mid (e, f) \in X_{a, b}\}\setminus E(\mathcal{A}).
  \]
If $a'\neq a$ or $b'\neq b$, then $X_{a, b}\neq X_{a', b'}$ and so 
$Y_{R_{a, b}R_{a, b}^{-1}}\neq Y_{R_{a', b'}R_{a', b'}^{-1}}$, as required. Finally, \(R_{a, b} \in Y_{\varnothing}\), but \(R_{a, b} \notin Y_{R_{a, b} R_{a, b}^{-1}}\), \(Y_{\varnothing} \neq
Y_{R_{a, b} R_{a, b}^{-1}}\) for any non-zero idempotent \(R_{a, b} R_{a, b}^{-1}\).
\end{proof}

\section{Graph inverse semigroups}
\label{section-gis}

In this section we give a full characterisation of when an arbitrary graph inverse semigroup is $E$-disjunctive.

For this purpose, we define a \textit{graph} $\Gamma = (\Gamma^0, \Gamma^1, \textbf{s}, \textbf{r})$ to be a quadruple consisting of two sets, $\Gamma^0$ and $\Gamma^1$, and two functions $\textbf{s}, \textbf{r} \colon \Gamma^1 \to \Gamma^0$, called the \textit{source} and \textit{range}, respectively. The elements of $\Gamma^0$ and $\Gamma^1$ are called \textit{vertices} and \textit{edges}, respectively. A sequence $p=e_1e_2\cdots e_k$ of (not necessarily distinct) edges $e_i \in \Gamma^1$, such that $(e_i)\textbf{r} = (e_{i + 1})\textbf{s}$ for $1\leq i\leq k - 1$, is a \textit{path} from $(e_1)\s$ to $(e_k)\rr$. We define $(p)\s = (e_1)\s$ and $(p)\rr = (e_k)\rr$, and refer to $k$ as the \textit{length} of $p$. The elements of $\Gamma^0$ are paths of length $0$, and we denote by $\Path(\Gamma)$ the set of all paths in $\Gamma$.  

Define the \textit{graph inverse semigroup $S(\Gamma)$ of a graph $\Gamma$} to be the inverse semigroup with zero $0\not\in \Gamma ^ 0 \cup \Gamma ^ 1$, generated by $\Gamma^0$ and $\Gamma^1$, together with a set of elements $\Gamma^{-1} = \{e^{-1} \mid e\in \Gamma^1\}$, that satisfies the following four axioms, for all $u, v\in \Gamma_0$ and $e, f\in \Gamma^1$:
\begin{description}
 \item[(V)] \(vv=v\) and $vu = 0$ if \(v\neq u\),
 \item[(E1)] $(e)\textbf{s}\ e = e\ (e)\textbf{r} = e$,
 \item[(E2)] $(e)\textbf{r}\ e^{-1} = e^{-1}\ (e)\textbf{s} = e^{-1}$,
 \item[(CK1)] $f^{-1}f = (f)\textbf{r}$ and $e^{-1}f = 0$ if \(e\neq f\).
\end{description}
(Note that ``V'' is for ``vertex'', ``E'' is for ``edges'', and ``CK'' is for ``Cuntz-Kreiger'' given the origins of the study of graph inverse semigroups in the study of Cuntz-Kreiger algebras and the similarity of CK1 to the Cuntz-Kreiger relations.)  

 For every $v\in \Gamma^0$ we define $v^{-1} = v$, and for every $q = e_1\cdots e_k\in \Path(\Gamma)$ we define $q^{-1} = e_k^{-1}\cdots e_1^{-1}$. It follows directly, by repeated application of (CK1), that every non-zero element in $S(\Gamma)$ can be written in the form $pq^{-1}$ for some $p, q\in \Path(\Gamma)$. It is routine to show that $S(\Gamma)$ is an inverse semigroup, with $(pq^{-1})^{-1} = qp^{-1}$ for every non-zero $pq^{-1}\in S(\Gamma)$. 

The congruences of a graph inverse semigroup were characterised in~\cite{Wang2019}.

A \emph{Wang triple} $(H,W,f)$ on $\Gamma$ consists of a set $H \subseteq \Gamma^0$ such that $H$ is closed under reachability in $\Gamma$ (i.e. if $u\in H$ and there is a path from $u$ to $v\in \Gamma^0$, then $v\in H$ also), a set $W \subseteq \{v \in \Gamma^0 \setminus H \mid |(v)\s_{\Gamma\setminus H}^{-1}| = 1\}$, and a \emph{cycle} function $f: C(\Gamma^0) \to \Z^+\cup\{\infty\}$ (where $C(A)$ is the set of cycles consisting of vertices in the set $A\subseteq \Gamma ^ 0$) such that $(c)f = 1$ for all $c \in C(H)$, $(c)f = \infty$ for all $c \notin C(H\cup W)$, and the restriction of $f$ to $C(W)$ is invariant under cyclic permutations. In~\cite{Wang2019}, the term ``congruence triple" is used for this concept.

Given a Wang triple $(H, W, f)$ on a graph $\Gamma$, we define the corresponding congruence to be the least congruence on $S(\Gamma)$ containing the following set:
\begin{equation}\label{eq-0}
\begin{array}{rcl}
 (H\times \{0\}) \cup \{(w, ee^{-1}) \mid w \in W,\ (e)\mathbf{s} = w,\ (e)\mathbf{r}\not\in H\} \\ \cup \{(c^{(c)f}, (c)\s) \mid c\in C(W), \ (c)f \in \Z^+\}.
 \end{array}
\end{equation}

Henceforth we identify the Wang triples and the congruences they represent.

An \textit{isolated vertex} in a graph $\Gamma$ is a vertex $v\in \Gamma ^ 0$ such that $v \neq (e)\s$ and $v\neq (e)\rr$ for all $e \in \Gamma ^ 1$. An \textit{out-edge} of
a vertex \(v \in \Gamma ^ 0\) is an edge \(e \in \Gamma ^ 1\) such that \((e)\s = v\).

    \begin{theorem}
      \label{thm:wang-triple-ip}
      A Wang triple \((H, W, f)\) of a graph inverse semigroup \(S(\Gamma)\) is an idempotent-pure congruence if and only if \(H\) is a set
      of isolated vertices and \((c)f = \infty\) for all \(c \in C(W)\).
    \end{theorem}

    \begin{proof}
        \((\Rightarrow)\): We prove the contrapositive. If \(H\) contains a vertex \(v\) that is not isolated, then there is  \(e\in \Gamma ^ 1\) such that \((e)\rr = v\) or \((e)\s = v\). Thus
        \((e, 0)\in (H, W, f)\). But $e$ is not an idempotent, and $0$ is idempotent, and so \((H, W, f)\) is not idempotent-pure.

        If \((c)f = x \in \Z^+\) for some $c\in C(W)$, then \((c^{x}, (c)\s) \in (H, W, f)\). Since
        \(c^x\) is not an idempotent and \((c)\s\) is an idempotent, \((H, W, f)\) is not idempotent-pure.

        \((\Leftarrow)\): We define
        \(A = (\{0\}\cup H)^2\)
        and we define $B$ to consist of the pairs \[((xp_1) (yp_1)^ {-1}, (xp_2)(yp_2)^{-1})\in S(\Gamma)\] such that 
        \( x, y, p_1,p_2\in \Path(\Gamma), (x)\rr = (p_1)\s= (p_2)\s = (y)\rr;
        (e)\s\in W,\) for all \( e \)  in \(p_1\) or \(p_2\).
        Let \(\rho= A \cup B \cup \{(x, x):x\in S(\Gamma)\}\). Note that \(\rho\) never relates an idempotent to a non-idempotent. It is thus sufficient to show that $\rho = (H, W, f)$. 
        If $(x, y)\in  A$,  then $x$ and \(y\) both belong to $H \cup \{0\}$, and so $(x, y) \in (H, W, f)$ also.
        
        If $((xp_1) (yp_1)^ {-1}, (xp_2)(yp_2)^{-1})\in B$, where $p_1 = e_1\cdots e_n$  for some $e_1, \ldots, e_n\in \Gamma ^ 1$, then 
    \[
    (xp_1(yp_1)^ {-1}, xy ^ {-1}) = (xp_1p_1 ^{-1}y^ {-1}, xy ^{-1}) = (xe_1\cdots e_n e_n ^ {-1} \cdots e_1 ^ {-1}y^ {-1}, xy ^{-1}).
    \]
    Since $(\s(e_i), e_ie_i ^ {-1}) \in (H, W, f)$, it follows that for every $i$
    \begin{multline*}
    (xe_1\cdots e_i e_i ^ {-1} \cdots e_1 ^ {-1}y^ {-1}, xe_1\cdots e_{i-1}(e_i)\s e_{i-1} ^ {-1} \cdots e_1 ^ {-1}y^ {-1}) \\
    = (xe_1\cdots e_i e_i ^ {-1} \cdots e_1 ^ {-1}y^ {-1}, xe_1\cdots e_{i-1} e_{i-1} ^ {-1} \cdots e_1 ^ {-1}y^ {-1})\in (H, W, f).
    \end{multline*}
    Hence, by transitivity, 
    \[
    (xe_1\cdots e_n e_n ^ {-1} \cdots e_1 ^ {-1}y^ {-1}, xy ^{-1}) \in (H, W, f).
    \]
    So \((xp_1(yp_1)^ {-1}, xy ^ {-1})\in  (H, W, f)\) and hence by symmetry \((xp_2(yp_2)^ {-1}, xy ^ {-1})\in (H, W, f)\). It follows that $\rho \subseteq (H, W, f)$.

    For the converse, it suffices to show that $\rho$ is a congruence, and that $\rho$ contains the pairs in \eqref{eq-0}. 

    To show that $\rho$ is transitive, we will check individually if \(A \circ B\),
    \(A \circ A\) and \(B \circ B\) are all contained in \(\rho\). For \(A \circ A\), if \((x, y),
    (y, z) \in A\), then \(x, z \in H \cup \{0\}\) and so \((x, z) \in A\). For \(A \circ B\), if
    \((z, xp_1 (yp_1)^ {-1}) \in A\) and \((xp_1 (yp_1)^ {-1}, xp_2 (y p_2)^{-1}) \in B\). As it lies in a pair in
    \(A\), \(xp_1 (yp_1)^{-1}\) is a
    vertex in \(H\) or \(0\). As \(0\) can never occur in a pair in \(B\), we can assume that \(xp_1 (yp_1)^{-1}\) is a vertex in \(H\). However, as the first entry in a pair in \(B\), \(xp_1 (yp_1)^ {-1}\) either contains a
    vertex in \(W\), or \(p_1\) is the empty path. The first case cannot happen as there is no path from
    \(\Gamma \setminus H\) to \(H\), and \(W \cap H = \varnothing\). In the second case, \(xp_1 (yp_1)^ {-1} = xy^{-1}\), which lies in
    \(A\). Thus \(xy^{-1}\) is a vertex in \(H\), and so \(x = y \in H\). As \(p_2\) is either empty or a
    path starting in \(H\) which intersects a vertex in \(W\), which never happens, we have that \(p_2\)
    is empty. Therefore \((z, xp_2 (yp_2)^{-1}) = (z, xy^{-1}) = (z, xp_1 (yp_1)^{-1})\in A\).
        
    For \(B \circ B\), suppose \(\mathbf{q} = (xp_1 (yp_1)^ {-1}, xp_2 (y p_2)^{-1}) \in B\) and \(\mathbf{r} = (wp_3 (zp_3)^ {-1}, wp_4 (z p_4)^{-1}) \in B\), where \(xp_2 (y p_2)^{-1} = wp_3 (zp_3)^ {-1}\). Either \(x\)
    is a prefix of \(w\) or \(w\) is a prefix of \(x\); without loss of generality suppose \(w\) is a
    prefix of \(x\). So \(x = wu\), for some path \(u\). It follows that \(p_3 = up_2\). So
    \(xp_1 = wup_1\) and \(\mathbf{q} = (wup_1 (yp_1)^{-1}, wup_2 (y p_2)^{-1})\). Additionally, as
    \(p_3 = u p_2\), \(\mathbf{r} = (wup_2 (zup_2)^ {-1}, wp_4 (z p_4)^{-1})\). We also have that
    \(y = zu\), as \(wup_2 (yp_2)^{-1} = wup_2 (zup_2)^{-1}\). Thus \(\mathbf{q} = (wup_1 (zup_1)^{-1}, wup_2 (zu p_2)^{-1})\)
    and \(\mathbf{r} = (wp_3 (zp_3)^ {-1}, wp_4 (z p_4)^{-1})\). It follows that \[(wup_1(zup_1)^{-1}, wp_4 (zp_4)^{-1}) \in B,\]
    as required.

    We now show that \(\rho\) is a congruence. If \((x, y)\in \rho\), then $(x ^ {-1}, y ^ {-1})\in \rho$, and so it suffices to show
    that \(\rho\) is a right congruence.
    If \((x, y) \in A\) and \(s \in \Gamma^0 \cup \Gamma^1 \cup (\Gamma^1)^{-1}\), then as \(H\) consists of isolated vertices, \(xs, ys \in H \cup \{0\}\), and so \((xs, ys) \in A\). If \((xp_1 (yp_1)^ {-1}, xp_2(yp_2)^{-1}) \in B\)
    and \(s \in \Gamma^0 \cup \Gamma^1 \cup (\Gamma^1)^{-1}\), then at least one of the following cases applies:
    \begin{enumerate}
        \item \(s\in \Gamma ^ 0\) and \(s =  (y)\s\). In this case, \((xp_1 (yp_1)^ {-1} s, xp_2(yp_2)^{-1} s) = (xp_1 (yp_1)^ {-1}, xp_2(yp_2)^{-1}) \in B\).
        \item \(s \in \Gamma ^ 0\) and \(s \neq (y)\s\). In this case \(((xp_1 (yp_1)^ {-1}s, xp_2(yp_2)^{-1}s) = (0, 0) \in A\).
        \item \(s \in \Gamma ^ 1\) and \((s)\s \neq (y)\s\). Again, \(((xp_1 (yp_1)^ {-1}s, xp_2(yp_2)^{-1}s) = (0, 0) \in A\).
        \item \(s \in \Gamma ^ 1\) and \(s\) is the first edge in \(y\). Then \(y = s y'\), for some
        path \(y'\) in \(\Gamma\), and so \(((xp_1 (yp_1)^ {-1}s, xp_2(yp_2)^{-1}s) = ((xp_1 (y'p_1)^ {-1}, xp_2(y'p_2)^{-1}) \in B\). 
        \item \(s \in \Gamma ^ 1\), \(y\notin \Gamma^0\), and \((s)\s = (y)\s\), \(s\) is not the first edge in \(y\). Again, 
        \[((xp_1 (yp_1)^ {-1}s, xp_2(yp_2)^{-1}s) = (0, 0) \in A.\]
        \item \(s \in \Gamma ^ 1\), \(y \in \Gamma^0\setminus W\)
        and \((s)\s = (y)\s\). From the definition of $B$, the source of each edge in
        \(p_1\) lies in \(W\), and \((p_1)\s = (y)\rr = y \notin W\). In particular, $p_1$ contains no edges, i.e. \(p_1 \in \Gamma^0\). Similarly, \(p_2 \in
        \Gamma^0\). Thus \(p_1 = p_2 = y\). So \(((xp_1 (yp_1)^ {-1}s, xp_2(yp_2)^{-1}s) = ((xy (y)^ {-1}s, xy(y)^{-1}s) = (xs, \ xs) \in \rho\).
        \item \(s \in \Gamma ^ 1\), \(y \in W\), $(s)\s = y$, and \(p_1, p_2 \notin\Gamma ^ 0\). Then by the choice of \(W\), $s$ is the unique edge with source $y$, and so \(p_1=sp_1'\) and \(p_2=sp_2'\) for some paths \(p_1', p_2'\) . Hence
        \[(xp_1 (yp_1)^ {-1}s, xp_2(yp_2)^{-1}s) =(xsp_1' p_1'^ {-1}, xsp_2'p_2'^{-1}) \in B.\]
        \item \(s \in \Gamma ^ 1\), \(y \in W\), $(s)\s = y$, and \(p_1, p_2 \in \Gamma ^ 0\). As
        \(s(p_1) = \s(p_2) = \rr(y) = y\), we have \(p_1 = y = p_2\), and so \((xp_1 (yp_1)^ {-1}, xp_2(yp_2)^{-1}) = (xp_1 (yp_1)^ {-1}, xp_1(yp_1)^{-1}) \in \rho\).
        \item \(s \in \Gamma ^ 1\) and \(y \in W\), $(s)\s = y$ and precisely one of \(p_1\) and
        \(p_2\) lies in \(\Gamma ^ 0\). Assume without loss of generality, that
        \(p_1\in \Gamma^0\). Then as \(W\) is part of a Wang triple, every
        vertex in \(W\) has a unique out-edge, and so \(s\) is the unique edge
        with source \(y\), and so \(p_2= sp_2'\) for some path \(p_2'\). 
        Hence
         \[(xp_1 (yp_1)^ {-1}s, xp_2(yp_2)^{-1}s) =(xs, xsp_2'p_2'^{-1}) \in B.\]
        \item \(s ^ {- 1} \in \Gamma ^ 1\) and \((s)\s = (y)\s\). Then let \(z = s^{-1} y\). Note
        \((z)\rr = (y)\rr\). In addition,
        \[
        ((xp_1 (yp_1)^ {-1}s, xp_2(yp_2)^{-1}s) = ((xp_1 (s^{-1}yp_1)^ {-1}, xp_2(s^{-1}yp_2)^{-1})
        = ((xp_1 (zp_1)^ {-1}, xp_2(zp_2)^{-1}) \in B.
        \]
        \item  \(s ^{-1}\in \Gamma ^ 1 \) and \((s)\s \neq (y)\s\). Again, \(((xp_1 (yp_1)^ {-1}s, xp_2(yp_2)^{-1}s) = (0, 0) \in A\). 
    \end{enumerate}
    Hence $\rho$ is a congruence, as required. 
    \end{proof}

    Next, we state and prove the main theorem of this section. 

    \begin{theorem}\label{Thm:graph_inverse}
        A graph inverse semigroup defined using a graph \(\Gamma\) is \(E\)-disjunctive if and only if
        \(\Gamma\) has no isolated vertices, and every vertex in \(\Gamma\) has either \(0\) or at
        least \(2\) out-edges.
    \end{theorem}
    
    \begin{proof}
      By \cref{thm:wang-triple-ip}, \(S(\Gamma)\) admits a non-trivial idempotent-pure
      congruence if and only if it has a (potentially empty) set \(H\) of isolated vertices and a set \(W \subseteq
      \{v \in \Gamma^0 \setminus H \mid |(v)\s_{\Gamma \setminus H}^{-1}| = 1\}\) such that there is
      a cycle function \(f\) on \(\Gamma\) such that \((c)f = \infty\) for all \(c \in C(W)\).

      \((\Rightarrow)\): Suppose that \(\Gamma\) has isolated vertices or there exists a vertex
      \(v \in \Gamma^0\) with $|(v)\s^{-1}| = 1$. In the former case, taking \(H\) to be the non-empty
      set of isolated vertices, and \(W = \varnothing\) gives a non-trivial idempotent-pure congruence. 
      In the latter case, \(H = \varnothing\) and \(W= \{v\in \Gamma^0: |(v)\s^{-1}| = 1\}\neq \varnothing\) gives a non-trivial idempotent-pure congruence. Thus \(S(\Gamma)\) is not \(E\)-disjunctive.

      \((\Leftarrow)\): Suppose that \(\Gamma\) has no isolated vertices and that every vertex has
      either \(0\) or \(2\) out-edges. For \(\Gamma\) to admit a non-trivial idempotent-pure congruence,
      we would have \(H = \varnothing\) and \(W = \varnothing\), as these are the only possibilities
      for \(H\) and \(W\). However, the congruence defined by this pair is equality, and so
      \(S(\Gamma)\) is \(E\)-disjunctive.
    \end{proof}

\section{Finite monogenic inverse monoids}
\label{section-monogenic}
In this section we characterise those finite monogenic inverse monoids that are $E$-disjunctive. Although we concentrate on monoids rather than semigroups in this section,  analogues of the main results hold for monogenic inverse semigroups, and these can be concluded from the results in this section together with \cref{cor:adj-1}. In order to do this, we require the following characterisation of finite monogenic inverse monoids, and some related results, mostly arising from 
\cite{Preston1986aa}. 

If $i_1, \ldots, i_m\in \{1, \ldots, n\}$, then we denote by 
\([i_1, \ldots, i_m]\) the element of the symmetric inverse monoid $I_n$ with domain $\{i_1, \ldots, i_{m - 1}\}$, image $\{i_2, \ldots, i_m\}$, and that maps $i_j$ to $i_{j + 1}$ for all $j\in \{1, \ldots, m - 1\}$ (and thus with the image of \(i_m\) not defined).

\begin{lem}[Lemma 8 in \cite{elliott2023counting}]\label{cor-chain-cycle}
        If $M$ is a finite monogenic inverse submonoid of \(I_n\), then there exist $a, b\in \N\setminus \{0\}$ such that \(a+b=n\) and
    $M$ is isomorphic to the inverse submonoid of $I_n$ generated by 
    \[x=[1, \ldots, a]\cup p\]
     \(p\) is some permutation on the set \(\{a+1, \ldots, a+b\}\). 
    Moreover, if $o(p)$ is the (group theoretic) order of $p$, then $M$ is isomorphic to the submonoid of \(I_{a+o(p)}\) generated by \([1, \ldots, a]\cup (a+1, \ldots, a+o(p))\).
\end{lem}

Let $n, k\in \N$ be such that $n\geq 0$ and $k\geq 1$ and let \(S_{n, k}\) be defined by the inverse monoid presentation
\[ \InvPres{x}{x^nx^{-n} =  x^{n+1}x^{-(n+1)},  x^nx^{-n}= x^nx^{-n}x^k }.\]

\begin{theorem}[Theorem 10 in \cite{elliott2023counting}]\label{thm-isomorphism}
  The inverse submonoid of \(I_{n+k}\) generated by a partial permutation
  \[[1, 2, \ldots , n]\cup (n+1, n+2, \ldots, n+k)\]
  is isomorphic to $S_{n, k}$
  for all \(n\geq 0\) and \(k\geq 1\).
  Moreover, if \(m\geq 0\) and \(l \geq 1\), then \(S_{n, k}\cong S_{m, l}\) if and only if \(n = m\) and \(k=  l\). 
\end{theorem}

\begin{lem}
  \label{lem:monogenic-idempotents}
  Let $n\geq 0$ and $k\geq 1$. Then 
  every element of \(S_{n, k}\) can be uniquely expressed in the form \(x^{-a}x^bx^{-b}x^c\) where \(a, c\leq b<n\) or \(x^nx^{-n}x^a\) where \(0\leq a<k\); and 
   \(
    E(S_{n, k}) =  \{x^{-a}x^bx^{-b}x^a \mid a\leq b < n\}
    \cup \{x^nx^{-n}\}.
  \)
\end{lem}
\begin{proof}
By \cite[Lemma 6]{elliott2023counting} (or \cite[Proposition 4]{Dyadchenko1984})
the set of words of the form \(x^{-a}x^ {b}x^{-b}x^c\) with \(b < n\) contains representatives for every element of $S_{n,k}$.
   In particular, \(x^{-a}x^bx^{-b}x^c\) when restricted to the set \(\{1, \ldots, n\}\) is the partial permutation
  \(\{(a+1, c+1), (a+2, c+2), \ldots, (a + (n-b), c+(n-b))\}\).
  Hence distinct words of the form  \(x^{-a}x^ {b}x^{-b}x^c\) represent distinct elements of $S_{n, k}$. 
  The remaining elements of the monoid are those in the ideal generated by \(x^nx^{-n}\). Since $S_{n, k}$ and the inverse monoid generated by $x=[1, \ldots, n](n+1, \ldots, n+k)$ coincide (by \cref{thm-isomorphism}), $x ^ {n}x ^ {-n}$ is the identity on $\{n+ 1, \ldots, n + k\}$, and so $x ^ {n} x ^ {-n} x = (n + 1, \ldots, n + k)$ generates a cyclic group of order $k$.

  The claim about idempotents is immediate from \cref{thm-isomorphism}.
\end{proof}

\begin{lem}
    \label{lem:monogenic-size}
    [Theorem 2 in \cite{Preston1986aa}]\label{size}
    If $n\geq 1$ and $k\geq 1$, then 
    \[|S_{n,k}| = \frac{n(n+1)(2n+1)}{6} + k.\]
\end{lem}
Note that $S_{0, k} = S_{1, k}$ and so $|S_{0, k}| = k + 1$.

The next theorem is the main result of this section, characterising the idempotent-pure congruences on $S_{n, k}$ when $n\geq 0$ and $k\geq 1$.

\begin{theorem}
\label{thm:ip-monogenic}
If \(\rho\) is a non-trivial congruence on \(S_{n, k}\), then $S_{n, k} / \rho \cong S_{n', k'}$,
where \(1\leq n'\leq n\) and \(k'|k\).
 Moreover, $\rho$ is idempotent-pure if and only if \(k'=k\) and \(n\leq k\).
\end{theorem}
\begin{proof}
Every homomorphic image of an inverse monoid is an inverse monoid, and since $S_{n, k}$ is monogenic, $S_{n, k} / \rho$ is monogenic also. In particular, $S_{n, k} / \rho$  is isomorphic to $S_{n', k'}$ , for some $n'\leq n$ and $k' \leq k$ (by \cref{cor-chain-cycle}, and \cref{thm-isomorphism}). Since $S_{n', k'}$ contains a cyclic subgroup of order $k'$ (by \cref{thm-isomorphism}) it follows that $k' | k$. 

We now show that if \(\rho\) is idempotent-pure, then \(k' = k\) and \(n \leq k\). Suppose \(\rho\) is idempotent-pure.
Since $k'| k$ it suffices to show that \(k' \geq k\).
Seeking a contradiction, suppose that \(k'<k\). Then  
\[[x^{n}x^{-n}]_\rho = [x^{n'}x^{-n'}]_\rho =  [x^{n'}x^{-n'}x^{k'}]_\rho = [x^{n}x^{-n}x^{k'}]_\rho.\]
But
\(x^{n}x^{-n}x^{k'} \notin E(S_{n, k})\) by \cref{lem:monogenic-idempotents}, a contradiction. It follows that $k' = k$.

It remains to show that \(n \leq k\). If $n > k = k'$, then since \(n' \leq n\), we can assume \(n' < n\), as otherwise \(\rho\) would be
trivial. So 
\[[x^{n-1}x^{-(n-1)}]_\rho=[x^{n'}x^{-n'}]_\rho =[x^{n'}x^{-n'}x^{k}]_\rho=[x^{n'}x^{-n'}x^{k'}]_\rho= [x^{(n-1)}x^{-(n-1)}x^{k}]_\rho,\]
where \(x^{(n-1)}x^{-(n-1)}x^{k} \notin E(S_{n, k})\) by \cref{lem:monogenic-idempotents}. Thus $\rho$ is not idempotent-pure, a contradiction. Hence $k' = k$ and $n \leq k$ as required.

For the converse, suppose that $k' = k$ and $n \leq k$. Then for all \(0\leq a, c\leq b<n\), \(1 \leq f \leq n\), and \(e, g < \max(n, k)\) 
\begin{align*}
    [x^{-a}x^bx^{-b}x^a]_\rho =[x^{-e}x^fx^{-f}x^g]_\rho 
    &\implies  (b=f  \text{ and } e=g=a) \text{ or } (b, f \geq n' \text{ and } -a+a=-e+g)\\
     &\implies (b=f  \text{ and } e=g=a) \text{ or } (b, f \geq n' \text{ and } e=g)\\
     &\implies e=g.
\end{align*}
Hence \(x^{-e}x^fx^{-f}x^g\) is an idempotent by \cref{lem:monogenic-idempotents}. Also
\[[x^nx^{-n}]_\rho =[x^{-e}x^fx^{-f}x^g]_\rho \implies  f\geq n'  \text{ and } k'|(g-e)\implies   k|(g-e)\implies g-e=0\implies   e=g
\]
so \(x^{-e}x^fx^{-f}x^g\) is again an idempotent, and \(\rho\) is idempotent-pure.
\end{proof}

Next, we use \cref{thm:ip-monogenic} to characterise the \(E\)-disjunctive finite monogenic inverse monoids.

\begin{cor}\label{cor:mon_E_dis}
     A finite monogenic inverse monoid is E-disjunctive if and only if it is isomorphic to \(S_{n, k}\) for some \(k, n\) with  \(n>k\) or \(n=1\). 
\end{cor}
\begin{proof}
     By \cref{thm:ip-monogenic}, a congruence $\rho$ on \(S_{n, k}\) is idempotent-pure if and only if \(k'=k\) and \(n\leq k\) (where \(S_{n,k}/\rho \cong S_{n', k'}\)).
     Let \(P(n, k)\) be the statement: for all \((n', k')\neq (n, k)\) with \(1\leq n'\leq n\), \(1\leq k'|k\), either \(k' \neq k\) or \(n > k\).
     So \(S_{n, k}\) is \(E\)-disjunctive if and only if \(P(n, k)\) is true.
We show that \(P(n,k)\) holds if and only if \(n>k\) or \(n=1\). 

\((\Rightarrow)\): Suppose \(P(n, k)\) is true. If \(n=1\), then the proof is complete. Otherwise set \(n'=1<n\) and \(k'=k\). Since \(P(n, k)\) holds, either \(k'\neq k\) or \(n>k\). Hence, since $k'= k$,  \(n>k\), as required.

\((\Leftarrow)\): If \(n>1\), then \(n>k\), and so \(P(n, k)\) holds immediately.
Otherwise if \(n=1\), then for all \(n'\leq n\) and \(k'|k\) with \((n', k')\neq (n, k)\), \(n'=n=1\) and so \(k'\neq k\). Hence in either case \(P(n, k)\) holds.
\end{proof}

We conclude this section by showing that, asymptotically, almost none of the monogenic inverse monoids are $E$-disjunctive.

\begin{cor}
     \label{cor:monogenic-proportion}
     The proportion of isomorphism classes of monogenic inverse monoids of size at most \(m\) which are \(E\)-disjunctive tends to \(0\) as \(m\) tends to infinity. 
\end{cor}

\begin{proof}
%
Let $m\in \Z_{\geq 1}$ be given.
By \cref{thm-isomorphism}, the number of monogenic inverse monoids up to isomorphism with size at most $m$ equals the number of $S_{n, k}$ such that $|S_{n, k}|\leq m$.
Similarly, by \cref{cor:mon_E_dis}, the number of $E$-disjunctive monogenic inverse monoids up to isomorphism with size at most $m$ equals 
the number of $S_{n, k}$ such that $|S_{n, k}|\leq m$ where $n > k$ or $n = 1$. In particular, it suffices to find the proportion of pairs $\{(n, k) \in \mathbb{Z}_{\geq 0} \times \mathbb{Z}_{\geq 1} \mid |S_{n, k}| \leq m\}$ such that $n > k$ or $n = 1$.
By \cref{lem:monogenic-size},
\[|S_{n, k}| = \frac{n(n+1)(2n+1)}{6} + k \leq m\quad \text{if and only if}\quad 
k\leq m- \sum_{i=0}^{n} i^2 = m -  \frac{n(n+1)(2n+1)}{6}.\]
For all \(j\in \Z_{\geq 1}\), there exists \(m_j\in \Z_{\geq 1}\) such that \(m_j\geq j\) and the number of $(n, k)$, such that \(|S_{n,k}| \leq m_j\) is greater than \(jm_j\). For example, if $j = 6$, then for \(m \geq 105\), the  
number of $(n, k)$ with $|S_{n, k}|\leq m$ is
\begin{align*}
    \sum_{n=1}^\infty\max\left(m- \sum_{i=0}^{n} i^2, 0\right) & \geq 
    \sum_{n =1}^ 6 \left( m- \sum_{i=0}^{n} i^2 \right)\\
    & = m + (m-1) + (m-5)+(m-14)+(m-30)+(m-55) \\
    & = 7m-105\geq 6m.
\end{align*}

We next find an upper bound for the number of $(n, k)$ such that $S_{n,k}$ is $E$-disjunctive and $|S_{n, k}| \leq m$. 
It suffices to give an upper bound for the number of pairs in $\{(n, k) \in \mathbb{Z}_{\geq 0} \times \mathbb{Z}_{\geq 1} \mid |S_{n, k}| \leq m\}$ such that $n > k$ or $n = 1$:

\begin{align*}
& \phantom{=} \sum_{n = 1} ^ {\infty} |\{ k\in\Z_{\geq 1} \mid (n > k \text{ or } n = 1) \text{ and } |S_{n, k}|\leq m \}| \\
& =
 |\{ k\in\Z_{\geq 1} \mid |S_{1, k}|\leq m \}|+\sum_{n = 2} ^ {\infty} |\{ k\in\Z_{\geq 1} \mid n > k \text{ and } |S_{n, k}|\leq m \}|
 \\
& =
 |\{ k\in\Z_{\geq 1} \mid k+1\leq m \}| +\sum_{n = 2} ^ {\infty} |\{ k\in\Z_{\geq 1} \mid n > k \text{ and } n(n + 1) (2n + 1)/6 +k \leq m \}|
\\
& =
 m - 1+\sum_{n = 2} ^ {\infty} |\{ k\in\Z_{\geq 1} \mid n > k \text{ and } n(n + 1) (2n + 1)/6 +k \leq m \}|
 \\
& =
 m - 1+\sum_{n = 2} ^ {\infty} |\{ k\in\Z_{\geq 1} \mid n > k \text{ and } k \leq m - n(n + 1) (2n + 1)/6 \}|
 \\
 & \leq
 m - 1+\sum_{n = 2} ^ {\infty} \min( \max(m - n(n + 1) (2n + 1)/6, 0), n - 1)\\
   & \leq
 m - 1+\sum_{n = 2} ^ {\infty} \min(\max(m - n^3/6, 0), n-1)\\
  & \leq
 m - 1+\sum_{n = 2} ^ {\lfloor \sqrt[3]{6m} \rfloor} \min(m - n^3/6, n-1)\\
 & \leq
 m - 1+\sum_{n = 2} ^ {\lfloor \sqrt[3]{6m} \rfloor}  n-1\\
  & \leq
 m - 1-\sqrt[3]{6m}+\frac{\sqrt[3]{6m}(\sqrt[3]{6m}+1)}{2}\\
   & \leq
 m +\sqrt[3]{6m}\sqrt[3]{6m}+1\\
  & \leq
 7m.
\end{align*}
Thus 
\begin{align*}\lim_{m\to \infty} \frac{\text{\# monogenic E-disjunctive inverse semigroups of size at most m}}{\text{\# monogenic inverse semigroups of size at most m}}
& \leq 
\lim_{n\to \infty}\frac{7m_n}{nm_n}\\
& =\lim_{n\to \infty}\frac{7}{n} = 0.\qedhere
\end{align*}
\end{proof}

\part{A structure theory for $E$-disjunctive inverse semigroups}

In this part of the paper we consider various structural properties of
$E$-disjunctive semigroups. In \cref{sec:ratio-idempotents} we find a bound on
the ratio of idempotent to non-idempotent elements in an $E$-disjunctive
semigroup; in \cref{sec:maximage} we consider the maximum $E$-disjunctive
homomorphic images of an inverse semigroup; in \cref{sec:preactions} we define a
notion we refer to as preactions which is used extensively in the final section
\cref{sec:mcealister}; where we prove that every inverse semigroup can be
defined in terms of a semilattice and an $E$-disjunctive inverse semigroup. 

\section{Ratio of idempotents to non-idempotents}
\label{sec:ratio-idempotents}
Roughly speaking semilattices are as far from being $E$-disjunctive as
possible. More specifically, every $E$-disjunctive homomorphic image of a
semilattice is trivial. In this section, we precisely formalise this notion by
showing that inverse semigroups with too many idempotents are not
\(E\)-disjunctive. In particular, we will prove the following theorem.

\begin{theorem}
  \label{idem_bound_thm}
  Let \(S\) be an \(E\)-disjunctive inverse semigroup and \(\kappa = |S\setminus E(S)|\). Then \(|S| \leq 2^\kappa + \kappa\).
\end{theorem}

If $S$ is a finite inverse semigroup, then we define $\mathbf{h}\colon S\to \N_{\geq
0}$ so that $(s)\mathbf{h}$ is the largest value in $\N_{\geq0}$ such that
there is a chain with maximum element $s$ of this length in the natural partial order of $S$.

In the following we require the notion of Green's relations on a semigroup $S$, which we briefly recall; see \cite[Chapter 2]{Howie} for more details.
Green's $\mathscr{R}$-relation is the equivalence relation on $S$ consisting of those pairs $(a, b)\in S\times S$ that generate the same principal right ideal of $S$, i.e. $a S ^ 1 = \{as \mid s\in S\} \cup \{a\} = b S ^ 1$. Green's $\mathscr{L}$-relation is the left ideal dual of $\mathscr{R}$, and Green's $\mathscr{D}$- is defined as $\mathscr{D} = \mathscr{L} \circ \mathscr{R}$.
If $a, b\in S$ and $a S^ 1\subseteq b S^{1}$, then we write $a \leq_{\mathscr{R}} b$, and likewise for $\leq_{\mathscr{L}}$.

\begin{lem}\label{lem-the-little-d}
  If $S$ is a finite inverse semigroup and $s, t \in S$ are such that
  $s\leq_{\mathscr{R}}t$ or $s\leq_{\mathscr{L}}t$, then $(s)\mathbf{h} \leq (t)\mathbf{h}$.
\end{lem}
\begin{proof}
  We prove the lemma in the case that $s\leq_{\mathscr{R}}t$, the proof in the
  other case is similar. Let $s'\in S$ be such that $ts' = s$. Suppose that $n
  = (t)\mathbf{h}$ and suppose that $t_1 \coloneqq  t, t_2, \ldots, t_{n} \in S$ are
  such that $t_i > t_{i + 1}$ for all $i$. We set $s_i = t_is'$ for all $i$.
  Since $t_{i + 1} \leq t_i$ implies $t_{i + 1} = (t_{i+ 1}t_{i+ 1} ^ {-1})t_i$ (by \cite[Proposition 5.2.1]{Howie}) it follows that
  \[
    s_{i + 1} = t_{i + 1}s' = (t_{i+ 1}t_{i+ 1} ^ {-1})t_is' = (t_{i + 1}t_{i +
    1} ^ {-1})s_i \leq s_i
  \]
  and so $(s)\mathbf{h} \geq (t)\mathbf{h}$.
\end{proof}

\begin{lem}
  Let \(S\) be a finite inverse semigroup. If \(s, \ t \in S\) are such
  that \(s\mathscr{D}t\), then \((s)\mathbf{h} = (t)\mathbf{h}\).
\end{lem}
\begin{proof}
  Since $s\mathscr{D}t$, there exists $u\in S$ such that $s\mathscr{R} u
  \mathscr{L} t$ and so $(s)\mathbf{h} = (u)\mathbf{h} = (t)\mathbf{h}$, by
  \cref{lem-the-little-d}. 
\end{proof}

\begin{lem}\label{fall_to_lower_level_lem}
Let \(S\) be a finite inverse semigroup and let $e, f\in S$ be distinct idempotents. If \((f)\mathbf{h} \leq (e)\mathbf{h}\), then \((ef)\mathbf{h} < (e)\mathbf{h}\).
\end{lem}
\begin{proof}
  If \(e\) and \(f\) are incomparable, then \(ef<e\) or \(ef <f\). In either case, \((e)\mathbf{h} \geq (f)\mathbf{h} > (ef)\mathbf{h}\).
  Otherwise, \(f < e\), and so \(ef < e\) and \((ef)\mathbf{h} < (e)\mathbf{h}\).
\end{proof}

If \(S\) is a finite inverse semigroup, then we define \(N(S)\) to be the set of idempotents in \(e\in S\) such that there exists a non-idempotent $u\in S$ such that $uu ^ {-1} = e$.
We also define \[\phi_{S} \colon E(S) \to \mathcal P(N(S))\]
by
\[(f)\phi_{S} = \makeset{e\in N(S)}{ \(e\leq f\)}.\]

\begin{lem}
\label{interesting_hom_lem}
If \(S\) is a finite inverse semigroup, then \(\phi_{S} \colon E(S) \to \mathcal P(N(S))\) 
defined by 
\[(f)\phi_{S} = \set{e\in N(S)}{e\leq f}\]
is a homomorphism where the operation on $\mathcal{P}(N(S))$ is $\cap$.
\end{lem}
\begin{proof}
  Let \(e, \ f\in E(S)\), and \(g\in N(S)\). Then
  \begin{align*}
      g\in (e)\phi_{S}\cap (f)\phi_{S} &\iff g\leq e\text{ and }g\leq f\\
      &\iff ge =g=gf\\
      &\iff gef=g\\
      &\iff g\leq ef\\
      & \iff g\in (ef)\phi_{S}.\qedhere
  \end{align*}
\end{proof}

If $I$ is an ideal of an inverse semigroup $S$, then the natural partial order on $I$ is just the intersection of the natural partial order of $S$ with $I\times I$.

\begin{lem}\label{semi_embed_lem}
If \(S\) is a finite \(E\)-disjunctive inverse semigroup, then \(\phi_{S}\colon E(S) \to \mathcal{P}(N(S))\) is an embedding.
\end{lem}

\begin{proof}
By \cref{interesting_hom_lem}, we need only show that \(\phi_{S}\) is injective.
  We proceed by induction on $(S)\mathbf{h} \coloneqq  \max\set{(s)\mathbf{h}}{s\in S}$.
  If \((S)\mathbf{h} = 0\), then $S = \varnothing$ and so \(\phi_{S}\) is an embedding.

  Suppose that \((S)\mathbf{h} = k\) and the result holds for all finite \(E\)-disjunctive inverse semigroups $T$ with $(T)\mathbf{h} < k$.
  Then  the set \(I \coloneqq  \set{s\in S}{(s)\mathbf{h} < k}\) is 
  an ideal of \(S\), and $(I)\mathbf{h} = k - 1$, and $I$ is \(E\)-disjunctive by \cref{ideal_lem}. 
  Thus by induction \(\phi_{I}\) is an embedding.

  Suppose  that there exist \(e, f \in E(S)\) such that \((e)\phi_{I}= (f)\phi_{S}\). 
  If \(e, f\in I\), then, since $\phi_{I}$ is just the restriction of $\phi_{S}$ to $I$, 
  \[(e)\phi_{I} = (e)\phi_{S} = (f)\phi_{S}= (f)\phi_{I}\]
  and so, since \(\phi_{I}\) is injective, \(e=f\), as required.

  Hence it remains to prove the lemma in the case that $e\not\in I$ or $f\not\in I$. Suppose without loss of generality that \(e\not \in I\) and, seeking a contradiction, that $e\neq f$.
  By assumption, 
  \begin{equation}\label{eq-ef-phi}
  (e)\phi_{S} = (e)\phi_{S}\cap (e)\phi_{S}=(e)\phi_{S}\cap (f)\phi_{S}=(ef)\phi_{S}.
  \end{equation}
  
 If there exists a non-idempotent $u\in S$ such that $e=uu^{-1}$, then \(e \in (e)\phi_{S}\). Since $e\neq f$ and \((e)\phi_{S} = (f)\phi_{S}\) by assumption, \(e < f\). It follows that $(f)\mathbf{h} > (e)\mathbf{h} = k = (S)\mathbf{h}$, which is a contradiction.
  
  Suppose that $uu^{-1}\neq e$ for all non-idempotents $u\in S$, and suppose that $\rho = \{(e, ef)\} \cup \Delta_S$.
  To reach our final contradiction it suffices to show that \(\rho\) is a congruence.
  We show \(\rho\) is a right congruence; the proof that \(\rho\) is a left congruence is symmetric.
  
  Let \(u\in S\) be arbitrary. If \(u = e\), then \(eu=ee=e\), $ef=efe=efu$, and $(e,ef) \in\rho $,
  and so $(eu,efu) \in \rho$. If \(u\neq e\), then 
  we will also show that \(eu= efu\).
  By assumption, $uu^{-1}\neq e$.
  Since \((e)\mathbf{h} = k\), it follows from \cref{fall_to_lower_level_lem}, that \(euu^{-1}, efuu^{-1} \in I\). So
  \begin{align*}
      (euu^{-1})\phi_{I} &= (euu^{-1})\phi_{S}\\
      &= (e)\phi_{S} \cap (uu^{-1})\phi_{S}\\
      &= (ef)\phi_{S}\cap (uu^{-1}) \phi_{S} && \text{by }\eqref{eq-ef-phi}\\
       &= (efuu^{-1})\phi_{S}\\
       &= (efuu^{-1})\phi_{I}.
  \end{align*}
  Thus \(euu^{-1} =efuu^{-1}\), since $\phi_{I}$ is an embedding and so $eu=efu$ also.
  Therefore \(\rho\) is a non-trivial
  idempotent-pure congruence on \(S\), and so \(S\) is not 
  \(E\)-disjunctive, a contradiction.
\end{proof}

The converse of \cref{semi_embed_lem} is not true, for example, if
\(S\) is the strong semilattice of groups defined by an identity map from the cyclic group \(C_2\) of order $2$ to \(C_2\), then \(\phi_{S}\) is injective,
but \(S\) is not \(E\)-disjunctive.

Let \(S\) be an inverse semigroup and let \(e \in E(S)\). Then \textit{syntactic
readout} of \(e\) is the function \(\phi_e \colon ((S\setminus E(S)) \cup
\{1_S\}) \times ((S\setminus E(S)) \cup \{1_S\}) \to \{0, \ 1\}\) defined by 
\[
(\alpha, \beta)\phi_e = 
\begin{cases}
  0 & \text{if }  \alpha e \beta \in E(S)\\
  1 & \text{if } \alpha e \beta \notin E(S)
\end{cases}
\]
for all $\alpha, \beta\in S\setminus E(S) \cup \{1\}$.

  \begin{lem}
      \label{lem:syn-read}
      Let \(S\) be an \(E\)-disjunctive inverse semigroup and let
      \(e, f \in E(S)\). Then $\phi_e = \phi_f$
      (i.e. $e$ and $f$
      have the same
      syntactic readout) if and only if \(e = f\).
  \end{lem}

\begin{proof}
The converse implication is trivial. 

    For the forward implication, since $S$ is $E$-disjunctive, it follows from \cref{lem:syn-eq} that the syntactic congruence on $S$ is $\Delta_S$. 
    Suppose that $e, f\in E(S)$ have the same syntactic readout.
    To show \(e = f\), it suffices to show that $(e, f)$ belongs to the syntactic congruence of $S$. In other words, 
 to show that 
    \begin{equation}\label{eq-readout}
        \alpha e \beta \in E(S) \iff \alpha f \beta \in E(S)
    \end{equation}
    for all \(\alpha, \beta \in S^1\).
    By assumption, \eqref{eq-readout} holds for all 
    \(\alpha, \beta \in S^1 \setminus E(S)\).

    Suppose that $\alpha\in E(S)$ or $\beta\in E(S)$. We may suppose 
    without loss of generality that \(\alpha \in E(S)\) and that $\alpha e\beta \in E(S)$.
    Since $\alpha\in E(S)$, \(\alpha e \beta = e \alpha \beta\) and \(\alpha f
    \beta = f \alpha \beta\). If \(\alpha \beta \in E(S)\),
    then \(f \alpha \beta = \alpha f \beta \in E(S)\). Otherwise,
    \(\alpha \beta \notin E(S)\), and since $(1_S, \alpha\beta)\phi_e = (1_S, \alpha\beta)\phi_f$ it follows that \(e \alpha
    \beta \in E(S)\) implies \(f \alpha \beta \in E(S)\). 
    Hence in both cases $\alpha e \beta \in E(S)$ implies $\alpha f \beta\in E(S)$, and the converse implication 
    follows by symmetry.
\end{proof}

\begin{proof}[Proof of \cref{idem_bound_thm}.]
We consider the cases when $S$ is finite and infinite separately. 

Suppose that \(S\) is finite. 
Clearly, \(|S| = |E(S)| + n\). By \cref{semi_embed_lem}, it follows that 
\[|S| = |E(S)| + n =|(E(S))\phi_{S}| + n \leq |\mathcal P(N(S))| + n.\]
The map sending a non-idempotent \(s\) of \(S\) to \(ss^{-1}\), is surjective with image set \(N(S)\). Thus \(|N(S)|\leq n\). In particular
\[|S|\leq |\mathcal P(N(S))| + n =  2^{|N(S)|} + n \leq 2^n + n.\]

Suppose that \(S\) is infinite, and that \(\kappa = |S \setminus E(S)|\). By \cref{lem:syn-read}, idempotents of \(S\) are uniquely determined by their syntactic readouts. There are at most \((2^\kappa)(2^\kappa) = 2 ^ {\kappa}\) syntactic readouts and so \(|S| \leq 2^\kappa + \kappa\).
\end{proof}

\begin{cor}
  \label{cor:nonidempotents_infinite}
  Let \(S\) be an infinite \(E\)-disjunctive inverse
  semigroup. Then \(S\) has infinitely many non-idempotents.
\end{cor}

The next example shows that the bound in \cref{idem_bound_thm} is sharp. 

\begin{ex}
  Let \(\kappa\) be a finite or infinite cardinal. We define a Clifford semigroup \(S\) with
  \(\kappa\) non-idempotents and \(2^\kappa\) idempotents by defining
  a strong semilattice of groups.
The semilattice $Y$ is the power set of \(\kappa\) under intersection; the groups are defined by:
\begin{equation*}
G_y =
\begin{cases}
    1 & \text{if } |y| \neq 1\\
    C_2 & \text{if } |y| = 1\\
\end{cases}
\end{equation*} 
for all $y\in Y$ where $1$ denotes the trivial group and $C_2$ the cyclic group of order $2$; and every homomorphism $\psi_{y, z}: G_y \to G_z$ with $y, z\in Y$ is constant; see \cref{attaining_idem_bound_fig} for an example.
Clearly there is an idempotent in $S$ for every subset of $\kappa$, and there is a non-idempotent for every element of $\kappa$. Hence $|S| = 2 ^ {\kappa} + \kappa$.
Since \(S\) is a semilattice of abelian groups, \(S\) is also
  commutative. 
      \begin{figure}
      \begin{tikzpicture}
        [scale=.4, auto=left,every node/.style={rectangle, draw}]
        
        \node (l3) at (0, 15) {\(1\)};
        
        \node (l21) at (-5, 10) {\(1\)};
        \node (l22) at (0, 10) {\(1\)};
        \node (l23) at (5, 10) {\(1\)};
        
        \node (C21) at (-5, 5) {\(C_2\)};
        \node (C22) at (0, 5) {\(C_2\)};
        \node (C23) at (5, 5) {\(C_2\)};
        
        \node (0) at (0, 0) {\(1\)};

        \draw (l3) to (l21);
        \draw (l3) to (l22);
        \draw (l3) to (l23);

        \draw (l21) to (C21);
        \draw (l21) to (C22);
        \draw (l22) to (C21);
        \draw (l22) to (C23);
        \draw (l23) to (C22);
        \draw (l23) to (C23);
        
        \draw (C21) to (0);
        \draw (C22) to (0);
        \draw (C23) to (0);
      \end{tikzpicture}
      \caption{Egg-box diagram of a semigroup that attains the bound
      in \cref{idem_bound_thm} when \(\kappa = 3\)}
      \label{attaining_idem_bound_fig}
    \end{figure}
  
  It remains to show that \(S\) is \(E\)-disjunctive. Let $e, f\in E(S)$ be such that \(e < f\). 
If $e$ and $f$ are the idempotents belonging to $G_y$ and $G_z$, respectively, then we abuse our notation by writing $\psi_{e,f}$ instead of $\psi_{y, z}$.
  By \cite[Theorem 6]{Clifford_IP}, it suffices to
  show that there exists \(g \in E(S)\) such that the map \(\psi_{eg,
  fg}\) is not injective. Since
  \(e < f\), \(f\) is non-zero. For this semigroup $S$, $N(S) = \set{h\in E(S)}{h\in G_y \text{ for some } y\in Y \text{ with } |y| = 1}$.
  If $f\in N(S)$, then we set $g = f > e$ and so $g\not \leq e$. If $f\not\in N(S)$, then there exists  
  $g\in N(S)$ such that $g\leq f$ and $g\not \leq e$ by the construction of $S$.
  In either case, $g\in N(S)$, $g\not\leq e$, and $g \leq f$.
  Hence \(eg < g\) and \(fg = g\). The former implies that $eg = 0$, and so \(\psi_{fg, eg}\) is a constant function from \(C_2\) to a trivial
  group, and is not injective, as required.
\end{ex}

\section{Maximum \(E\)-disjunctive images}
    \label{sec:maximage}

    Every inverse semigroup $S$ has an \(E\)-disjunctive quotient: the quotient of $S$ by its syntactic congruence $\rho$. By \cref{lem:largest-ip-cong}, if $T$ is any $E$-disjunctive semigroup such that $T$ is a homomorphic image of $S$, then $T$ is a quotient of $S/ \rho$. As such we refer to $S/ \rho$ as the \textit{maximum \(E\)-disjunctive quotient} of the inverse semigroup $S$. Since quotients and homomorphic images are interchangeable, we may also refer to $S/\rho$ as the \textit{maximum \(E\)-disjunctive image} of $S$. We will show that many properties can be exchanged between an inverse semigroup and its maximum \(E\)-disjunctive image. The situation is somewhat
     similar to the relationship between an \(E\)-unitary inverse semigroup and its maximum group homomorphic image.
  Of course, unlike the case for maximum group images, maximum \(E\)-disjunctive images are not always groups, and so it
  is not clear to what extent this can be used to study inverse semigroups in general. 
  In this section we will explore the relationship between an inverse semigroup and its maximum $E$-disjunctive image. 
  In \cref{sec:mcealister}, we show that every inverse semigroup is described by its maximum \(E\)-disjunctive image and semilattice of idempotents, somewhat analogous to the description of $E$-unitary inverse semigroup via McAlister triples. 

  An inverse semigroup is \(E\)-unitary if and only if its minimum group congruence is idempotent-pure. Since every group is \(E\)-disjunctive, it follows that the maximum \(E\)-disjunctive image of an \(E\)-unitary inverse semigroup is a group. Conversely, if \(S\) is an inverse semigroup whose maximum \(E\)-disjunctive homomorphic image is a group, then the maximum group and $E$-disjunctive images coincide. In other words,  we have the following result.
  
  \begin{proposition}
  The maximum \(E\)-disjunctive image of an inverse semigroup $S$ is a group if and only if $S$ is \(E\)-unitary.
  \end{proposition}
  Next, we use a theorem of Kambites~\cite{Kambites2015} to show that a finitely generated inverse
  semigroup is finite if and only its maximum \(E\)-disjunctive image is finite. To do this,
  we first define the idempotent problem of a finitely generated inverse semigroup $S$ as follows.
    Let $\Sigma$ be a finite generating set for \(S\). Then the \textit{idempotent problem} of \(S\), with respect to \(\Sigma\), denoted \(\IP(S, \Sigma)\) is the language consisting of all words over \(\Sigma \cup \Sigma^{-1}\) that represent idempotents in \(S\).

  \begin{theorem} [\cite{Kambites2015}]
    \label{thm:inv_sgp_anisimov}
    Let \(S\) be an inverse semigroup with a finite generating set \(\Sigma\). Then \(S\) is finite
    if and only if \(\IP(S, \Sigma)\) is a regular language.
  \end{theorem}

\cref{thm:inv_sgp_anisimov} allows us to characterise the inverse semigroups with finite 
maximum $E$-disjunctive image as follows.

  \begin{theorem}\label{thm:fg-max-E-im}
    Let \(S\) be a finitely generated inverse semigroup. Then \(S\) has a finite maximum \(E\)-disjunctive image if and
    only if \(S\) is finite.
  \end{theorem}

  \begin{proof}
  ($\Leftarrow$):
    Since quotients of finite semigroups are finite, if \(S\) is finite, then so is its maximum \(E\)-disjunctive image. 
    
  ($\Rightarrow$):
    Suppose that \(S\) has a finite maximum \(E\)-disjunctive image $T$, that $\Sigma$ is a finite generating set for $S$, and that
  \(\phi \colon S \to T\) is an idempotent-pure epimorphism. Since \(\phi\) is idempotent-pure and surjective, \(\phi\) induces a bijection between \(E(S)\) and \(E(T)\) and an isomorphism from the free monoid $(\Sigma\cup \Sigma ^ {-1}) ^ *$ with generators $\Sigma\cup \Sigma ^ {-1}$ and the free monoid $((\Sigma\cup \Sigma ^ {-1})\phi) ^ *$
    that maps 
    \(\IP(S, \Sigma)\) to \(\IP(T, (\Sigma)\phi)\). By
    \cref{thm:inv_sgp_anisimov}, since $T$ is finite, \(\IP(T, (\Sigma)\phi)\) is a regular language. Thus \(\IP(S, \ \Sigma)\), as the image under an isomorphism of a regular language, is itself regular. Applying \cref{thm:inv_sgp_anisimov} again yields that \(S\) is finite.
  \end{proof}

  Every  semilattice has trivial maximum \(E\)-disjunctive image. Since every
  finitely generated semilattice is finite, if $S$ is an infinite semilattice,
  then $S$ is not finitely generated, but its maximum $E$-disjunctive image is
  finite. In other words, the finitely generated hypothesis in
  \cref{thm:fg-max-E-im} cannot be removed.


\begin{lem}\label{injR}
        Let \(\phi \colon S \to T\) be an idempotent-pure homomorphism. Then \(\phi |_R\) is
    injective for every \(\mathscr{R}\)-class \(R\) of $S$.
\end{lem}
\begin{proof}
    Suppose that \(a, b\in S\) are such that \(a\mathscr{R} b\) and 
    suppose that \((a)\phi = (b)\phi\).  By Green's Lemma the map $\lambda\colon R_a\to R_{a^{-1}a}$ defined by left multiplying by $a ^ {-1}$ is a bijection. Hence 
    \[
    (a^ {-1}a)\phi = (a^ {-1})\phi \cdot (a)\phi = (a^ {-1})\phi \cdot (b)\phi = (a^ {-1}b)\phi
    \]
    and $a^{-1}a \mathscr{R} a^{-1}b$ (since $\mathscr{R}$ is a left congruence).
    Hence we may assume without loss of generality that $a$ is an idempotent. 
    
    Since $\phi$ is idempotent-pure and $(a)\phi = (b)\phi$, it follows that $b$ is an idempotent. Hence since $a\mathscr{R} b$ this implies $a = b$. 
\end{proof}

\section{Preactions}
\label{sec:preactions}
In this section we define a notion that is a weakening of the notion of an
inverse semigroup action. This idea is somewhat analogous to the notion of
partial actions introduced in \cite{KellendonkLawson}. We require this somewhat
technical section in order to prove a generalisation of McAlister's
$P$-Theorem~\cite{McAlister} in \cref{sec:mcealister}. This generalisation
describes every inverse semigroup in terms of an \(E\)-disjunctive inverse
semigroup and a semilattice. 

If \(\mathcal Y\) is a subset of a poset \(\mathcal X\), then we write
\(\mathcal{Y}\mathord{\downarrow} = \{x\in \mathcal X\mid \exists y\in \mathcal
Y,\ x\leq y\}\) for the order ideal of \(\mathcal X\) generated by \(\mathcal
Y\).

Recall that a \textit{partial function} from \(X\) to \(Y\) is a function from a
subset of \(X\) to a subset of \(Y\). We will generalise the notation \(f \colon X \to Y\) to denote a partial function from \(X\) to \(Y\).
\begin{dfn}[\textbf{Action.}]\label{dfn-action}
Suppose that \(S\) is an inverse semigroup, that \(\mathcal Y\) is a poset (we view an unordered set as a poset in which all elements are incomparable when needed), and that \(\alpha \colon \mathcal Y\times S \to \mathcal Y\) is a partial function.
If $s\in S$, then we define 
\(\underline{s}_\alpha \colon \mathcal Y\to \mathcal Y\) to be the partial function defined by \((y)\underline{s}_\alpha = (y,s)\alpha\).  We write $\underline{s}_\alpha$ rather than $s_\alpha$, to avoid having to write parentheses, for example, we write $\underline{st}_\alpha$ instead of $(st)_\alpha$.
We say that \(\alpha\) is an \textit{action of \(S\) on \(\mathcal Y\)} if the following hold for all $s, t\in S$:
\begin{enumerate}
    \item the partial function \(\underline{s}_\alpha\) is an order isomorphism between subsets of \(\mathcal Y\);
    \item \(\underline{st}_\alpha=\underline{s}_\alpha \underline{t}_\alpha\) and \(\underline{s}_\alpha^{-1}=\underline{s^{-1}}_\alpha\).
\end{enumerate}
\end{dfn}

\begin{dfn}[\textbf{Preaction.}]\label{dfn-preaction}
Suppose that \(S\) is an inverse semigroup, that \(\mathcal Y\) is a poset, and that \(\alpha \colon \mathcal Y\times S \to \mathcal Y\) is a partial function.
If $s\in S^1$, then we define 
\(\underline{s}_q \colon \mathcal Y\to \mathcal Y\) to be the partial function defined by \((y)\underline{s}_q=(y,s)q\) (using the identity function if \(s\) is the adjoined identity in \(S ^ 1\)). 
We say that \(q\) is a \textit{preaction of $S$ on $\mathcal Y$} if the following hold for all $s, t, u\in S^1$:
\begin{enumerate}
    \item \label{defn:preaction_preservesorder}
    the partial function \(\underline{s}_q\) is an order isomorphism between subsets of \(\mathcal Y\);
    \item \label{dfn:preaction_downwards_agreement} 
    if \(s\leq t\), then \(\underline{s}_q\subseteq \underline{t}_q\);
    \item \label{dfn:preaction_connnected_groupoid*}
    if \((x, y)\in \underline{s}_q\) and \((y, z)\in \underline{t}_q\), then \((x, z)\in \underline{st}_q\), and  if \((x, y)\in \underline{s}_q\) then \((y, x)\in \underline{s^{-1}}_q\).
    \item \label{defn:down_closed condition} 
    $\dom(\underline{s}_q)\mathord{\downarrow} = \dom(\underline{s}_q)$;
    \item \label{defn:preaction_uses_Y} for all \(y\in \mathcal Y\), there is \(e\in E(S)\) such that \(y\in \dom(\underline{e}_q)\).
\end{enumerate}
\end{dfn}

\begin{lem}\label{lem-dfn-preaction}
Suppose that \(S\) is an inverse semigroup, that \(\mathcal Y\) is a poset, and that \(q \colon \mathcal Y \times S \to \mathcal Y\) is a preaction. If $s, t\in S ^ 1$ and $e\in E(S)$, then the following hold:

\begin{enumerate}[\rm (1)]
\addtocounter{enumi}{5}
    \item \label{defn:preaction_compositioncontainement}
    \(\underline{s}_q \underline{t}_q \subseteq \underline{st}_q\);
    \item \label{dfn:preaction_connnected_groupoid**}
If  $y, z\in \mathcal Y$, $s, t\in S$, and 
any two of \((x, y)\in \underline{s}_q\), \((y, z)\in \underline{t}_q\), \((x, z)\in \underline{st}_q\) hold, then so does the third;
\item \label{dfn:preaction-partial-id}
If \(e\in E(S)\) is an idempotent, then \(\underline{e}_q\) is a partial identity function.
\end{enumerate}
\end{lem}
\begin{proof} 
    \noindent \eqref{defn:preaction_compositioncontainement} This is an immediate consequence of \eqref{dfn:preaction_connnected_groupoid*}.
    \smallskip

\noindent\eqref{dfn:preaction_connnected_groupoid**}
        If  $y, z\in \mathcal Y$, $s, t\in S$,  
 \((x, y)\in \underline{s}_q\), and \((y, z)\in \underline{t}_q\),
then we have \((x, z)\in \underline{st}_q\) by \eqref{dfn:preaction_connnected_groupoid*}.
If  \((x, y)\in \underline{s}_q\) and \((x, z)\in \underline{st}_q\), then by \eqref{dfn:preaction_connnected_groupoid*}, \((y, x)\in \underline{s^{-1}}_q\) and so \((y, z)\in \underline{s^{-1}st}_q\). Thus by \eqref{dfn:preaction_downwards_agreement}, \((y, z)\in \underline{t}_q\).
If  \((y, z)\in \underline{t}_q\) and \((x, z)\in \underline{st}_q\), then by \eqref{dfn:preaction_connnected_groupoid*}, \((z, y)\in \underline{t^{-1}}_q\) and so \((x, y)\in \underline{stt^{-1}}_q\). Thus by \eqref{dfn:preaction_downwards_agreement}, \((x, y)\in \underline{s}_q\).
    \smallskip

\noindent \eqref{dfn:preaction-partial-id} From \eqref{defn:preaction_compositioncontainement}, we have
        that \(\underline{e}_q \underline{e}_q \subseteq \underline{e}_q\), and so \((x) \underline{e}_q = x\) for
        all \(x \in \im(\underline{e}_q)\). By \eqref{dfn:preaction_connnected_groupoid*}, \(\dom \underline{e}_q = \im(\underline{e^{-1}}_q) = 
        \im(\underline{e}_q)\), and so \(\underline{e}_q\) is 
        a partial identity function.
\end{proof}
    If $q \colon \mathcal Y \times S \to \mathcal Y$ in \cref{dfn-preaction} is an inverse semigroup action, then $q$ satisfies \cref{dfn-preaction}\eqref{defn:preaction_preservesorder}, \eqref{dfn:preaction_downwards_agreement}, and  \eqref{dfn:preaction_connnected_groupoid*} and all conditions in \cref{lem-dfn-preaction}.
\begin{ex}\label{ex:need_X}
    Suppose that $\min \N = 0$.
    We define a preaction \(q\) of the additive group \(\Z\) on the set \(\mathcal Y \coloneqq  -\N\) as follows:
    \[\dom(q)=\set{(n, z)\in (-\N) \times \Z}{n+z \in (-\N)},\quad 
    \text{and}\quad(n, z)q=n+z.\]
    In this example, if $z\in \Z$, then $(n)\underline{z}_q = n + z$ and $\dom(\underline{z}_q) = \set{x\in \Z}{ x \leq -z }$. 
    Since we are using additive notation for $\Z$, the conditions in
    \cref{dfn-preaction} become additive; for example, \eqref{dfn:preaction_connnected_groupoid*} becomes 
     ``if \((x, y)\in \underline{s}_q\) and \((y, z)\in \underline{t}_q\), then \((x, z)\in \underline{s+t}_q\)''.
    It is routine to verify that satisfies \cref{dfn-preaction}.
    However, $q$ is clearly not an action in the usual sense, however it naturally extends to one.
\end{ex}

The main result in this section is the following result, which, roughly speaking, states that every preaction can be extended to an inverse semigroup action, albeit on a larger set. 

\begin{theorem}\label{thm:extend_preactions}
Let \(\mathcal Y\) be a poset and let \(S\) be an inverse semigroup. If
\(q\colon \mathcal Y\times S \to \mathcal Y\) is a preaction, then there is a
poset \(\mathcal X_q\supseteq \mathcal Y\) and an action (by partial order
isomorphisms) \(\alpha_q \colon \mathcal X_q \times S \to \mathcal X_q\) such
that 
    \begin{enumerate}[\rm (1)]
        \item \(\mathcal Y\) is an order ideal of \(\mathcal X_q\);
        \item  the restriction of $\alpha_q$ to $(\mathcal Y\times S)\cap (\mathcal Y)\alpha_q^{-1}$ equals $q$;
        \item if \(a,b\in \mathcal X_q\), then \(a\leq b\) if and only if there is \(s\in S^1\) such that \((a,s)\alpha_q, (b,s)\alpha_q\in \mathcal Y\) and \((a,s)\alpha_q\leq (b,s)\alpha_q\).
    \end{enumerate}
\end{theorem}


      Roughly speaking, we define a set \(\mathcal X'_q\) to consist of the pairs in \(\mathcal Y
      \times S\) that we want to lie in the domain of \(\alpha_q\). In particular, if
      we can act on a point using an element \(s\) via \(\alpha_q\), we should
      be able to act on it with every left divisor of \(s\) first, in order for
      composition to work properly. The proof has the following steps: 
      \begin{itemize}
      \item  we define $\mathcal X'_q$ in \eqref{eq-defn-X-q-prime}; 
      \item 
      we define a function $\alpha_q' \colon \mathcal X_q' \times S \to \mathcal X_q'$ of $S$ on
       $\mathcal X_q'$ in \eqref{eq-defn-alpha-q-prime};
      \item prove that $\alpha_q'$ is an action on $\mathcal X_q'$ in \eqref{alpha-q-action};
       \item we define a preorder $\preceq$ on $\mathcal X_q'$ in \cref{preceq};
       \item  we show that the action $\alpha_q'$ preserves the preorder $\preceq$ in \cref{preceq-preserve}.
       \item we define the partially ordered set $\mathcal{X}_q$ to be the quotient of $\mathcal{X}_q'$ by the equivalence classes of $\preceq$ with the partial order induced by $\preceq$.
       \item we define an order-embedding \(\phi\) of $\mathcal Y$ (from \cref{thm:extend_preactions}) into \(\mathcal{X}_q\) in \eqref{de-phi} and \cref{phi-well-defined}. It follows that $\mathcal{Y}$ can be identified with an (order-isomorphic) subset of $\mathcal{X}_q$.
       \item we show that the domain of $q \colon \mathcal{Y}\times S \to \mathcal{Y}$ is downwards-closed in $\mathcal{X}_q$ under the partial order induced by $\preceq$ in \cref{claim-phi-leq}.
       \end{itemize}
        At that point we will have the necessary preliminaries to be able to give the proof of \cref{thm:extend_preactions}.

      We define:
      \begin{equation}\label{eq-defn-X-q-prime}
        \mathcal X_q'=\makeset{(y, s)\in \mathcal Y\times S}{there exists \(s'\in sS\) with \((y, s')\in \dom(q)\)}
      \end{equation}
      and $\alpha_q' \colon \mathcal X_q'\times S \to \mathcal X_q'$ by 
    \begin{equation}\label{eq-defn-alpha-q-prime}
    ((y,s),t)\alpha_q'=(y,st)
    \end{equation}
    if and only if $(y, s)\in \mathcal X_q'$ satisfies
    $s\in S t ^ {-1}$ and $(y, st) \in \mathcal X'_q$.
    
    As we did in \cref{dfn-preaction}, if $t\in S$, then we define $\underline{t}_{\alpha_q'}\colon \mathcal X_q'\to \mathcal X_q'$ by $(y, s) \underline{t}_{\alpha_q'} = ((y, s), t)\alpha_q'$ for all $(y, s)\in \mathcal X_q'$.
    With this notation
    \[\dom(\underline{t}_{\alpha_q'})=\makeset{(y, s)\in \mathcal X_q'}{\(s\in St^{-1}\) and \((y,st)\in \mathcal X_q'\)}.\]

    \begin{lem}
        \label{alpha-q-action}
        $\alpha_q'$ is an inverse semigroup action.
    \end{lem}

    \begin{proof}
    It suffices to verify the domains:
    \begin{align*}
        \dom(\underline{s}_{\alpha_q'})\cap (\dom(\underline{t}_{\alpha_q'}),s^{-1})\alpha_q'&=\makeset{(y, v)\in \mathcal X_q'}{\(v\in Ss^{-1}\), \((y,vs)\in \mathcal X_q'\), \(v\in St^{-1}s^{-1}\) and \((y, vst)\in \mathcal X_q'\)}\\
        &=\makeset{(y, v)\in \mathcal X_q'}{ \(v\in St^{-1}s^{-1}\) and \((y, vst)\in \mathcal X_q'\)}\\
        &=\dom(\underline{st}_{\alpha_q'}).\qedhere
    \end{align*}
    \end{proof}

\begin{lem}
    \label{preceq}
    If we define $\preceq$ on $\mathcal X_q'$ by 
    \((y_1,s_1)\preceq (y_2,s_2)\) if there exists \(s_3\in S ^ 1\) with
    \[(y_1,s_1s_3), (y_2,s_2s_3)\in \dom(q) \text{ and }(y_1,s_1s_3)q\leq (y_2,s_2s_3)q,\]
    then $\preceq$ is a preorder.
\end{lem}

\begin{proof}  
By the definition of \(\mathcal X_q'\), \(\preceq\) is reflexive. It remains to show that \(\preceq\) is transitive. Suppose that \((y_1,s_1), (y_2,s_2), (y_3,s_3)\in \mathcal X_q'\) and there are \(s_4, s_5\in S\) such that
\begin{eqnarray}
\label{eq-claim-1}
(y_1,s_1s_4), (y_2,s_2s_4)\in \dom(q) &\text{ and }(y_1,s_1s_4)q\leq (y_2,s_2s_4)q\\
\label{eq-claim-2}
(y_2,s_2s_5), (y_3,s_3s_5)\in \dom(q) &\text{ and }(y_2,s_2s_5)q\leq (y_3,s_3s_5)q.
\end{eqnarray}
It suffices to show that
\[(y_1,s_1s_5), (y_3,s_3s_5)\in \dom(q) \text{ and }(y_1,s_1s_5)q\leq (y_3,s_3s_5)q.\]
By \eqref{eq-claim-2}, it is thus sufficient to show that 
\[(y_1,s_1s_5)\in \dom(q) \text{ and }(y_1,s_1s_5)q\leq (y_2,s_2s_5)q.\]

As \((y_2,s_2s_4)\in \dom(q)\) and \((y_2,s_2s_5)\in \dom(q)\), it follows from \cref{dfn-preaction}\eqref{dfn:preaction_connnected_groupoid*} that \((y_2,s_2s_4)q\in \dom(\underline{s_4^{-1}s_5}_q)\) and hence by
\cref{dfn-preaction}\eqref{dfn:preaction_downwards_agreement} and 
\cref{lem-dfn-preaction}\eqref{defn:preaction_compositioncontainement}:
\[((y_2,s_2s_4)q, s_4^{-1}s_5)q=(y_2,s_2s_4 s_4^{-1}s_5)q=(y_2,s_2s_5)q.\]
Moreover as \(\dom(\underline{s_4^{-1}s_5}_q)\) is an order ideal by \cref{dfn-preaction}\eqref{defn:down_closed condition}, it follows that \(((y_1,s_1s_4)q, s_4^{-1}s_5)\in \dom(q)\) and hence 
\[((y_1,s_1s_4)q, s_4^{-1}s_5)q=(y_1,s_1s_5)q.\]
Since $(y_1, s_1s_4)q \leq (y_2, s_2s_4)q$ and $\underline{s_4^{-1}s_5}_q$ is $\leq$-preserving (by \cref{dfn-preaction}\eqref{defn:preaction_preservesorder}), it follows that 
\[
(y_1,s_1s_5)q = ((y_1,s_1s_4)q, s_4^{-1}s_5)q \leq ((y_2,s_2s_4)q, s_4^{-1}s_5)q = (y_2,s_2s_5)q,\]
as required.
\end{proof}

    \begin{lem}
        \label{preceq-preserve}
         The action \(\alpha_q'\) is \(\preceq\)-preserving.
    \end{lem}

    \begin{proof}
    If $t\in S$ and \((y_1,s_1), (y_2,s_2)\in \dom(\underline{t}_{\alpha_q'})\) and there is \(s_3\in S^1\) with
    \((y_1,s_1s_3)q\leq (y_2,s_2s_3)q\), then \(s_1, s_2\in St^{-1}\). Hence
    \[(y_1,s_1tt^{-1}s_3)q=(y_1,s_1s_3)q\leq (y_2,s_2s_3)q=(y_2,s_2tt^{-1}s_3)q\]
    so \(((y_1, s_1), t)\alpha_q' = (y_1, s_1t)\preceq (y_2, s_2t)= ((y_2, s_2), t)\alpha_q' \), and the action \(\alpha_q'\) is \(\preceq\)-preserving.
    \end{proof}

    We write $(y_1, s_1)\sim (y_2, s_2)$ if $(y_1, s_1) \preceq (y_2, s_2)$
    and $(y_2, s_2) \preceq (y_1, s_1)$, and denote by $[(y, s)]_{\sim}$ the $\sim$-equivalence class of $(y, s) \in \mathcal X_q$. 
    Let \(\mathcal X_q\) be the quotient of \(\mathcal X_q'\) by the equivalence relation $\sim$. Then $\mathcal X_q$ is partially ordered by $[(y_1, s_1)]_{\sim} \leq [(y_2, s_2)]_{\sim}$ if $(y_1, s_1) \preceq (y_2, s_2)$.

    We define $\phi\colon \mathcal Y \to \mathcal X_q$ by 
    \begin{equation}\label{de-phi}
    (y)\phi = [(z, u)]_{\sim} \text{ if }(y)q^{-1} \subseteq [(z, u)]_{\sim}.
    \end{equation}

    \begin{lem}
        \label{phi-well-defined}
         $\phi$ is a well-defined order-embedding and $\dom(\phi) = \mathcal Y$.
    \end{lem}

    \begin{proof}    
    If $(y)\phi = [(z, u)]_{\sim}$ and $(y)\phi = [(z', u')]_{\sim}$,
    then without loss of generality $(z, u)q = y = (z', u')q$ and so 
    $(z, u) \sim (z', u')$, meaning that $\phi$ is well-defined.
    We show that there is \((z, u)\in \mathcal X_q'\) such that \((z, u)q=y\). By \cref{dfn-preaction}\eqref{defn:preaction_uses_Y} we can pick $z=y$ and $u=e$ for some idempotent \(e\) such that $y\in \dom(e_q)$. This implies that the domain of $\phi$ is $\mathcal Y$. If $y_1, y_2\in \mathcal Y$, and $(z_1, u_1)\in (y_1)q ^{-1}$ 
    and $(z_2, u_2)\in (y_2)q ^{-1}$ for some $z_1, z_2, u_1, u_2$, then
    \begin{align*}
        y_1\leq  y_2 &\iff (z_1, u_1)q \leq (z_2, u_2)q\\
        &\phantom{\Leftarrow}\Rightarrow (z_1, u_1)\preceq (z_2, u_2)\\
        &\iff [(z_1, u_1)]_\sim \leq [(z_2, u_2)]_{\sim}\\
       &\iff y_1\phi \leq y_2\phi.
    \end{align*}
    
    Thus to conclude both that \(\phi\) is injective and order-preserving it suffices to show that \((z_1, u_1)\preceq (z_2, u_2) \) implies that \((z_1, u_1)q\leq (z_2, u_2)q \). 
    By the definition of $\preceq$, there exists \(s_3\in S\) such that \((z_1, u_1s_3)q\leq (z_2, u_2s_3)q\). By assumption, for $i\in \{1, 2\}$, $z_i\in \dom(\underline{u_i}_q)$ and $z_i\in \dom(\underline{u_is_3}_q)$. In other words, \((z_i, (z_i) \underline{u_i}_q) \in \underline{u_i}_q\) and
    \((z_i, (z_i)\underline{u_i s_3}_q) \in \underline{u_i s_3}_q\).
    Then \cref{lem-dfn-preaction}\eqref{dfn:preaction_connnected_groupoid**} tells us 
    \(((z_i)\underline{u_i}_q, (z_i)\underline{u_i s_3}_q) \in \underline{s_3}_q\). In particular, \(z_i \in \dom(\underline{u_i}_q \underline{s_3}_q)\).
    Hence 
\begin{align*}
(z_1, u_1)q & =  (z_1)\underline{u_1}_q & &\text{definition of }\underline{u_1}_q\\
& =  (z_1)\underline{u_1}_q \underline{s_3}_q \underline{s_3}_{q}^{-1} && \underline{s_3}_q \underline{s_3}_{q}^{-1} \text{ is the identity on }\dom(\underline{s_3}_q)\\
& =  (z_1)\underline{u_1s_3}_{q} \underline{s_3}_{q}^{-1} \\
& =  (z_1, u_1s_3)q \cdot \underline{s_3}_q ^{-1} \\
& \leq  (z_2, u_2s_3)q \cdot \underline{s_3}_q ^{-1} && \underline{s_3}_q\text{ is a partial order-isomorphism of } \mathcal Y\\
& =   (z_2, u_2)q.\qedhere
\end{align*}
    \end{proof}
    
    In light of \cref{phi-well-defined}, we may identify $\mathcal{Y}$ with its image under $\phi$, and define the partial order on \(\mathcal Y\) to be that 
    induced by the preorder \(\preceq\) on \(\mathcal X_q'\). We abuse notation
    by using \(\preceq\) to denote this partial order on \(\mathcal X_q\).

\begin{lem}\label{claim-phi-leq}
  The domain of \(q\) is downwards closed; that is, if $(y_1, s_1) \in \dom(q)$, and \((y_2, s_2)\preceq (y_1, s_1)\), then \((y_2, s_2)\in \dom(q)\).
\end{lem}
\begin{proof}
Let \((y_1, s_1) \in \dom(q)\) and \((y_2, s_2) \preceq (y_1, s_1)\). Then
by the definition of $\preceq$ there exists \(s_3 \in S^1\) such that
\[
(y_1, s_1 s_3), (y_2, s_2s_3)\in \dom(q),
\quad (y_2, s_2 s_3)q \leq (y_1, s_1s_3)q.
\]
Since \((y_1, (y_1)\underline{s_1}_q) \in \underline{s_1}_q\) and \((y_1, (y_1)\underline{s_1s_3}_q) \in \underline{s_1s_3}_q\), it follows by \eqref{dfn:preaction_connnected_groupoid*} that \(((y_1)\underline{s_1}_q, (y_1)\underline{s_1 s_3}_q) \in \underline{s_3}_q\)
\((y_1)\underline{s_1}_q \underline{s_3}_q = (y_1)\underline{s_1s_3}_q\); see \cref{fig:my_label}. In particular, \(((y_1)\underline{s_1}_q, s_3) \in \dom( q)\). 
Moreover \(((y_1) \underline{s_1}_q, s_3)q = (y_1, s_1s_3)q \geq (y_2, s_2 s_3)q\). By  \eqref{dfn:preaction_connnected_groupoid*}, \(\im(\underline{s_3}_q) = \dom(\underline{s_3^{-1}}_q) \) which is an order ideal by \eqref{defn:down_closed condition} and so \(((y_2, s_2 s_3)q) \in \dom (\underline{s_3^{-1}}_q)\). Applying \eqref{dfn:preaction_connnected_groupoid*} to 
$(y_2, (y_2,s_2s_3)q)\in \underline{s_2s_3}_q$ and \[((y_2,s_2s_3)q, ((y_2,s_2s_3)q, s_3^{-1})q) \in \underline{s_3^{-1}}_q,\] we obtain $(y_2,  ((y_2,s_2s_3)q, s_3^{-1})q)\in \underline{s_2s_3s_3^{-1}}_q$.

Thus $(y_2, s_2s_3s_3^{-1}) \in \dom(q)$ and, by \eqref{dfn:preaction_downwards_agreement},  
$\underline{s_2s_3s_3^{-1}}_q \subseteq \underline{s_2}_q$. In particular, 
$(y_2, s_2s_3s_3^{-1})q = (y_2, s_2)q$ and so
$(y_2,s_2) \in \dom(q)$.
\end{proof} 

A particular case of \cref{claim-phi-leq} is the following.

\begin{cor}\label{claim-phi-equiv-defined}
If $(y, s) \in \dom(q)$, then $[(y, s)]_\sim \subseteq \dom(q)$ and \((y, s) q\phi = [(y, s)]_\sim=(y, s)q q^{-1}\).  In other words,
 $[(y, s)]_{\sim}\phi ^ {-1} = (z, t)q$ for any $(z,t)\in [(y, s)]_{\sim}$. 
\end{cor}

\begin{proof}[Proof of \cref{thm:extend_preactions}]
We first establish part (1).
Let \([(y_2, s_2)]_{\sim} \leq [(y_1, s_1)]_{\sim}\in (\mathcal Y)\phi\).
Then by the definition of \(\phi\), there is \((y, s)\in \dom(q)\) with \([(y_1, s_1)]_{\sim}=[(y, s)]_{\sim}\).
Then by \cref{claim-phi-equiv-defined}, \((y_1, s_1)\in \dom(q)\). 
The assumption that \([(y_2, s_2)]_{\sim} \leq [(y_1, s_1)]_{\sim}\)
implies that  \((y_2, s_2)\preceq (y_1, s_1)\). Hence by \cref{claim-phi-leq}, \((y_2, s_2)\in \dom(q)\) and so \([(y_2, s_2)]_{\sim}=((y_2, s_2)q)\phi\in (\mathcal Y)\phi\), and we have shown
part (1).

    Define the partial function \(\alpha_q \colon \mathcal X_q \times S \to \mathcal X_q\) by
    \[([(y, u)]_\sim, s)\alpha_q= [((y, u), s)\alpha_q']_{\sim} 
    = [(y, us)]_\sim\]
where 
    \begin{align*}
        \dom(\alpha_q)&= \makeset{([(y, u)]_\sim, s)\in \mathcal X_q \times S}{\((y, u)\in \dom(s_{\alpha_q'})\)}\\
        &= \makeset{([(y, u)]_\sim, s)\in \mathcal X_q \times S}{\(u\in Ss^{-1}\) there is \(v\in S\) with \((y,usv)\in \dom(q)\)}.
    \end{align*}

\begin{figure}
\begin{tikzcd}
 & y_1 \arrow[dl, swap, "\underline{s_1s_3}_q"]\arrow[d, "\underline{s_1}_q"]\\
(y_1)\underline{s_1s_3}_q &  (y_1)\underline{s_1}_q \arrow[l, dashed, "\underline{s_3}_q"]\\
\rotatebox[origin=c]{90}{$\leq$} & \\
(y_2)\underline{s_2s_3}_q\arrow[dr, dashed, "\underline{s_3^{-1}}_q", swap]&  y_2 \arrow[l, "\underline{s_2s_3}_q"]\arrow[d, dashed, "s_2"]\\
& (y_2)\underline{s_2s_3}_q \underline{s_3^{-1}}_q 
\end{tikzcd}
\caption{Diagram demonstrating the arguments in the second claim within the proof of \cref{thm:extend_preactions}. The dashed lines indicate an application of \eqref{dfn:preaction_connnected_groupoid*} starting with the arrow labelled $\underline{s_3}_q$ and proceeding anti-clockwise.}
\label{fig:my_label}
\end{figure}

    It remains to check that the induced action of \(S\) on the copy $(\mathcal Y)\phi$ of \(\mathcal Y\) contained in \(\mathcal X_q\) is isomorphic to the action of $S$ on $\mathcal Y$.     
    To this end we define $\phi\oplus\id_S \colon \mathcal Y\times S \to \mathcal X_q\times S$ by
    \[
    (y, s) \phi\oplus\id_S = ((y)\phi, s),
    \]
    and 
    \[
    \mathcal Z = ((\mathcal Y)\phi\times S)\cap ((\mathcal Y)\phi)\alpha_q^{-1}.
    \]
    That is we will show:
    \[(\phi\oplus \operatorname{id}_S)\circ \alpha_q|_{\mathcal Z}\circ \phi^{-1}
    = q.\]
    We define the function on the left hand side by $Q$.
    Suppose that \((y, s)\in \dom(Q)\). It follows that \(y\in \dom(\phi)\). From \cref{dfn-preaction}\eqref{defn:preaction_uses_Y}, there exists \(e_y\in E(S)\) be such that \((y, e_y)\in \dom(q)\). Then
    \begin{align*}
        (y, s)Q&=(y, s)(\phi\oplus \operatorname{id}_S)\circ \alpha_q|_{Z}\circ \phi^{-1}\\
        &=([(y, e_y)]_\sim, s) \alpha_q |_{Z}\circ \phi^{-1}\\
          &=([(y, e_ys)]_\sim )\phi^{-1} && \text{by \cref{claim-phi-equiv-defined}}\\
           &=(y, e_ys)q \\
           &=(y, s)q && \text{by \cref{dfn-preaction}\eqref{dfn:preaction_downwards_agreement}.}
    \end{align*}
    Thus \(q |_{\dom(Q)} = Q\). If \((y, s)\in \dom(q)\), then, by the sequence of equalities above (in reverse order), $(y, s)\in \dom(Q)$. Hence $q = Q$. The map \(\alpha_q\) in the statement of the lemma can now be chosen by redefining \(\mathcal X_q:= (\mathcal X_q\setminus \im(\phi)) \cup \mathcal Y\) and the map \(\alpha_q\) by
    \[(y, s)\alpha_q= \begin{cases}
        (y, s)\alpha_q & y\notin \mathcal Y\\
        ((y)\phi, s)\alpha_q & y\in \mathcal Y.\\
    \end{cases}\]
    As we previously showed that \(Q=q\), it follows that now \(q=\alpha_q|_{(\mathcal Y\times S)\cap (\mathcal Y)\alpha_q^{-1}}\) and so part (2) of the theorem holds.
    
    Since part (3) of the theorem implies that $\alpha_q$ acts by partial order-isomorphisms on $\mathcal X_q$, the proof will be concluded by showing that part (3) holds. 
    Suppose that $[(y_1, s_1)]_{\sim}, [(y_2, s_2)]_{\sim} \in \mathcal X_q$.
    We must show that
    $[(y_1, s_1)]_{\sim}\leq [(y_2, s_2)]_{\sim}$ if and only if there exists $t\in S^1$ such that $ ([(y_1, s_1)]_{\sim}, t)\alpha_q \leq ([(y_2, s_2)]_{\sim},t)\alpha_q \in (\mathcal Y)\phi$. 

    By the definition of $\leq$, $[(y_1, s_1)]_{\sim}\leq [(y_2, s_2)]_{\sim}$ if and only if \((y_1, s_1)\preceq (y_2, s_2)\) if and only if (from the definition of $\preceq$) there exists \(t\in S^1\) such that \((y_1, s_1t), (y_2, s_2t)\in \dom(q)\) and  \((y_1, s_1t)q\leq (y_2, s_2t)q\in \mathcal Y\).
    This holds if and only if there exists \(t \in S^1\) such that 
     \(([(y_1, s_1)]_{\sim}, t)\alpha_q=[(y_1, s_1t)]_{\sim} \in (\mathcal Y)\phi\), \(([(y_2, s_2)]_{\sim}, t)\alpha_q=[(y_2, s_2t)]_{\sim} \in (\mathcal Y)\phi\), and \(([(y_1, s_1)]_{\sim}, t)\alpha_q\leq ([(y_2, s_2)]_{\sim}, t)\alpha_q\), as required. Hence part (3) of the theorem holds, and the proof of \cref{thm:extend_preactions} is complete at last.
\end{proof}

The next corollary follows immediately from \cref{thm:extend_preactions} and essentially states that preactions are precisely certain restrictions of certain actions.

\begin{cor}
    If \(S\) is an inverse monoid, \(\mathcal Y\) is a poset, and \(q \colon \mathcal Y\times S \to \mathcal Y\) is a partial function, then \(q\) is a preaction if and only if \(q\) satisfies \cref{dfn-preaction}\eqref{defn:down_closed condition} and \eqref{defn:preaction_uses_Y} and there is an action
    \(\alpha_q\) of \(S\) on a poset \(\mathcal X\supseteq \mathcal Y\) such that
\(q=\alpha_q|_{(\mathcal Y\times S)\cap (\mathcal Y)q^{-1}}\).
\end{cor}

\section{The \(Q\)-theorem}
\label{sec:mcealister}
The goal of this section is to introduce a means of defining an inverse semigroup in terms of a natural action of an \(E\)-disjunctive semigroup on a poset, and also to show that every inverse semigroup can be defined this way (\cref{thm:Q-semis-are-semis} and \cref{thm:Q-thm}). This construction generalises that of McAlister triples for \(E\)-unitary inverse semigroups
\cite{McAlister}. This theorem was already known and proved in
\cite{OCarroll75}. The authors of the present paper only discovered
\cite{OCarroll75} at a late stage of the preparation of this paper, and proved
the characterisation independently. We hope that the reader will find the presentation here  friendlier, and more modern than that in \cite{OCarroll75}.

Recall that if \(\mathcal{Y}\) is a subset of a poset \(\mathcal{X}\), then we write \(\mathcal{Y}\downset 
= \{x \in \mathcal{X} \mid \exists y\in \mathcal{Y},\ x \leq y\}\). If $\mathcal{Y} = \mathcal{Y}\downset$, then we say that $\mathcal{Y}$ is an \textit{order ideal} in $\mathcal{X}$.
\begin{dfn}[\textbf{$Q$-semigroup}]\label{dfn-q-semigroup}
Suppose that \(T\) is an inverse semigroup acting on a poset \(\mathcal{X}\) by partial order isomorphisms and \(\mathcal{Y}\) is a meet subsemilattice and order ideal of \(\mathcal{X}\) such that the following conditions hold:
\begin{enumerate}[(1)]
\item For all \(t \in T\), \(\dom t = (\dom t)\downset\);
 \item For all \(y \in \mathcal{Y}\) there exists \(t \in T\)
with \[\dom(t|_{\mathcal{Y}}) = \bigcap\set{\dom(s|_{\mathcal{Y}})}{s\in T
\text{ such that } y\in \dom(s|_{\mathcal{Y}})}.\] 
We denote the set $\dom(t|_{\mathcal{Y}})$ by $\delta(y)$.
\item For all \(x\in \mathcal{X}\), there is \(t\in T\) such that \((x)t\in \mathcal{Y}\).
\end{enumerate}
Then we define $Q(T, \mathcal{Y}, \mathcal{X})$ to be 
\[
    Q(T, \mathcal{Y}, \mathcal{X}) = \set{(y,t)\in \mathcal{Y}\times T}{ \dom (t) = \Dom(y),\ (y)t \in \mathcal{Y}},
\]
with multiplication defined by \((y_1,t_1) \cdot (y_2,t_2) = (((y_1) t_1 \wedge y_2) t_1^{-1} , t_1 t_2)\).
\end{dfn}

We note that part (2) in \cref{dfn-q-semigroup}, can be reformulated as:
\begin{enumerate}[(1)*]
\addtocounter{enumi}{1}
\item for all $y\in \mathcal Y$, there exists $t\in T$ such that $\dom(t|_\mathcal Y)$ contains \(y\) and is the least possible with respect to containment. 
\end{enumerate}

Note that 
$\Dom \colon \mathcal{Y} \to \mathcal{P}(\mathcal{Y})$ as defined in \cref{dfn-q-semigroup} is a homomorphism (of semilattices) where the operation on $\mathcal{P}(\mathcal{Y})$ is $\cap$.

If $(G, \mathcal{Y}, \mathcal{X})$ is a McAlister triple, then it will turn out that the inverse semigroups $P(G, \mathcal{Y}, \mathcal{X})$ 
and  $Q(G, \mathcal{Y}, \mathcal{X})$ coincide.

The main theorems of this section are the following; which we prove in \cref{subsection-q-semigroups-are-semigroups} and \cref{section-proof-of-q-thm}, respectively.

\begin{theorem}
\label{thm:Q-semis-are-semis}
If \((T, \mathcal{Y}, \mathcal{X})\) satisfy the conditions in \cref{dfn-q-semigroup}, then $Q(T, \mathcal{Y}, \mathcal{X})$ is an inverse semigroup. 
\end{theorem}
A converse of \cref{thm:Q-semis-are-semis} also holds.
\begin{theorem}
\label{thm:Q-thm}
Every inverse semigroup \(S\) is isomorphic to some $Q(T, \mathcal{Y}, \mathcal{X})$ from \cref{dfn-q-semigroup}, where $T$ is the maximum $E$-disjunctive homomorphic image of \(S\), and \(\mathcal Y\) is the semilattice of idempotents of \(S\).
\end{theorem}

\subsection{$Q$-semigroups are inverse semigroups}
\label{subsection-q-semigroups-are-semigroups}

In this section we give the proof of \cref{thm:Q-semis-are-semis}.

\begin{proof}[Proof of \cref{thm:Q-semis-are-semis}]
Let \(Q = Q(T, \mathcal{Y}, \mathcal X)\). We begin by showing that the multiplication defined in \cref{dfn-q-semigroup} is well-defined. 
Let \((y_1, t_1), (y_2, t_2) \in Q\). 
Then we must show \( (((y_1) t_1 \wedge y_2) t_1^{-1} , t_1 t_2)\in Q\); that is
\(\dom(t_1 t_2) = \Dom(((y_1) t_1 \wedge y_2) t_1^{-1})\) and
\(((y_1) t_1 \wedge y_2)t_1^{-1} t_1 t_2 \in \mathcal{Y}\).
 Since \((y_1) t_1 \wedge y_2 \leq y_1 t_1 \in \dom (t_1^{-1})\). Since \(\dom(t_1^{-1})\) is an order ideal, it follows that \((y_1)t_1 \wedge y_2\in \dom(t_1^{-1})\).
Similarly, $(y_1) t_1 \wedge y_2 \leq y_2\in \dom(t_2)$ and so 
$(y_1) t_1 \wedge y_2 \in \dom(t_2)$. Hence 
$((y_1) t_1 \wedge y_2)t_1^{-1} t_1 t_2 
= ((y_1) t_1 \wedge y_2) t_2\leq (y_2)t_2\in \mathcal{Y}$.
 
 It remains to show that 
 \[\Dom(((y_1)t_1 \wedge y_2) t_1^{-1})= \dom(t_1t_2).\]
 First note that for all \(t\in T\) and \(y\in \dom(t)\) we have 
 \begin{align*}
     (\Dom(y))t&=\left(\bigcap_{\substack{s\in T\\y\in \dom(s)}} \dom(s) \right)t\subseteq \bigcap_{\substack{s\in T\\y\in \dom(s)}} \dom(s)t
     =\bigcap_{\substack{s\in T\\y\in \dom(s)}} \dom(t^{-1}s)\subseteq \bigcap_{\substack{s\in T\\yt\in \dom(s)}} \dom(s) = \Dom((y)t).
 \end{align*}
Similarly, \(\Dom((y)t)t^{-1}\subseteq \Dom(y)\) and \(\Dom((y)t)=\Dom((y)t)t^{-1}t\subseteq (\Dom(y))t\).
It follows that
\begin{equation}\label{eq-whatever}
(\Dom(y))t=\Dom(yt).
\end{equation}
Hence 
\begin{align*}
    \Dom(((y_1)t_1 \wedge y_2) t_1^{-1})&=\Dom((y_1)t_1 \wedge y_2) t_1^{-1}\\
    &=(\Dom((y_1)t_1)\cap \Dom (y_2)) t_1^{-1}\\
    &=(\Dom(y_1)t_1\cap \Dom (y_2)) t_1^{-1}\\
    &=(\dom(t_1)t_1\cap \dom (t_2)) t_1^{-1}\\
    &=\dom(t_1)\cap \dom (t_2)t_1^{-1}\\
      &=\dom(t_1t_2).
\end{align*}
 
Next we prove that $Q$ is a semigroup. It suffices to show that the multiplication is associative.  
Let \((y_1,t_1),(y_2,t_2),(y_3,t_3) \in Q\). Let
\[
u = (((y_1) t_1 \wedge y_2) t_2 \wedge y_3) t_2^{-1} \text{ and }
v = (y_1) t_1 \wedge ((y_2) t_2 \wedge y_3) t_2^{-1}.
\]
We must first show that \(u = v\), which we do by showing \(u \leq v\) and \(v \leq u\). To this end  
    \[ u  = (((y_1) t_1 \wedge y_2) t_2 \wedge y_3) t_2^{-1} 
     \leq ((y_2) t_2 \wedge y_3) t_2^{-1}
     \]
     and 
     \[
    u  = (((y_1) t_1 \wedge y_2) t_2 \wedge y_3) t_2^{-1} 
     \leq  ((y_1) t_1 \wedge y_2) t_2 t_2^{-1} 
     = (y_1) t_1 \wedge y_2 
     \leq (y_1) t_1.
    \]
Thus \(u \leq v\). To show \(v \leq u\), it is sufficient to show that \((v)t_2 \leq (u) t_2\). We have
\begin{align*}
    (v) t_2 & = ((y_1) t_1 \wedge ((y_2) t_2 \wedge y_3) t_2^{-1}) t_2 
     \leq ((y_1) t_1 \wedge (y_2) t_2 t_2^{-1}) t_2 
     = ((y_1) t_1 \wedge y_2) t_2
     \end{align*}
     and
\begin{align*}
    (v) t_2 & = ((y_1) t_1 \wedge ((y_2) t_2 \wedge y_3) t_2^{-1}) t_2 
    \leq ((y_2) t_2 \wedge y_3) t_2^{-1} t_2 
    = (y_2) t_2 \wedge y_3.
\end{align*}
Hence $v\leq u$ and \(u = v\), as required. From the definition of the multiplication,
\begin{align*}
((y_1,t_1) \cdot (y_2,t_2)) \cdot (y_3,t_3) & = (((y_1) t_1 \wedge y_2) t_1^{-1}) , t_1 t_2) \cdot (y_3, t_3) \\
& = ((((y_1) t_1 \wedge y_2) t_1^{-1}) t_1 t_2 \wedge y_3) t_2^{-1} t_1^{-1}, t_1 t_2 t_3) \\
& = ((((y_1) t_1 \wedge y_2) t_2 \wedge y_3) t_2^{-1} t_1^{-1}, t_1 t_2 t_3) && ((y_1)t_1\wedge y_2)t_1^{-1}t_1 = (y_1)t_1\wedge y_2\\
& = ((u) t_1^{-1}, t_1 t_2 t_3) \\
& = ((v) t_1^{-1}, t_1 t_2 t_3) \\
& = (((y_1) t_1 \wedge ((y_2) t_2 \wedge y_3) t_2^{-1}) t_1^{-1} , t_1 t_2 t_3) \\
& = (y_1, t_1) \cdot ((y_2 t_2 \wedge y_3) t_2^{-1} , t_2 t_3) \\
& = (y_1,t_1)\cdot ((y_2,t_2) \cdot (y_3,t_3)). 
\end{align*}

We conclude the proof by showing that $Q$ is an inverse semigroup. 
We first show that \(Q\) is regular. Let \((y, t) \in Q\). We will
show that \((yt, t^{-1}) \in Q\), and that this is an inverse for \((y, t) \in Q\).
Since \((\Dom(y))t = \Dom(yt)\) by \eqref{eq-whatever}, 
\[\dom t^{-1} = (\dom t) t = (\Dom(y))t= \Dom(yt).\]
Therefore \((yt, t^{-1})
\in Q\). In addition,
\begin{align*}
    (y,t)\cdot ((y)t ,t^{-1})\cdot (y,t) & = (((y)t \wedge (y)t) t^{-1}, tt^{-1})\cdot (y, t) \\\
    & = ((y)tt^{-1}, tt^{-1})\cdot (y, t) \\
    & = (y, tt^{-1})\cdot (y, t) \\
    & = (((y) tt^{-1} \wedge y)tt^{-1}, tt^{-1} t) \\
    & = (y, t),
\end{align*}
and so \(Q\) is regular. It now suffices to show that the idempotents commute.
If \((y,t)\in Q\) is an idempotent, then \(t \in E(T)\). Conversely if \(e \in E(T)\) then \((y, e)\cdot (y, e)
= (((y)e \wedge y)e^{-1}, e^2) = (y, e)\). So $E(Q) = \set{(y, e)\in Q}{y\in \mathcal{Y}, e\in E(T)}$. These elements commute:
\[
    (y_1, e_1)\cdot (y_2, e_2) = (((y_1) e_1 \wedge y_2) e_1^{-1}, e_1 e_2)
    = (y_1 \wedge y_2, e_1 e_2)
    = (y_2, e_2)\cdot (y_1, e_1). \qedhere
\]
\end{proof}

\subsection{The proof of \cref{thm:Q-thm}}\label{section-proof-of-q-thm}

For the remainder of this section, we suppose that $S$ is a fixed inverse semigroup. 
The idea behind the proof of \cref{thm:Q-thm} is as follows. 
The semigroup $S$ has a quotient by an idempotent-pure congruence that is an \(E\)-disjunctive inverse semigroup; see \cref{sec:maximage}.
By \cref{injR}, an element \(s\in S\) is determined by its image in this quotient together with the idempotent $ss^ {-1}$. 
The set $\mathcal{Y}$ in the definition of $Q(T, \mathcal{Y}, \mathcal{X})$ is the set of idempotents $E(S)$ of $S$, and $T$ is the quotient of $S$ by its maximum idempotent-pure congruence $\rho$. 
We will prove that the function $S \to Q(T, \mathcal{Y}, \mathcal{X})$ defined by 
\[
s \mapsto (ss^{-1}, s/\rho)
\]
is an isomorphism for the correct choice of $\mathcal{X}$.

Roughly speaking, in order to capture the multiplication of $S$ in the definition of $Q(T, \mathcal{Y}, \mathcal{X})$, we need to be able to recover the idempotents $(st)(st)^{-1}$ from the idempotents $ss^{-1}$ and $tt^{-1}$. 
This is where the action comes in. The action we define comes from the conjugation (inverse semigroup) action of $S$ on $E(S)$ which is defined as follows.
We define \(\alpha \colon E(S)\times S \to E(S)\) by 
\begin{equation}\label{eq-action}
(e, s)\alpha= s^{-1}es\quad \text{ and } \quad \dom(\alpha)= \set{(e, s)\in E(S)\times S}{e\leq ss^{-1}}.
\end{equation}
It is routine to verify that this is an inverse semigroup action.
We also define \(\phi_{\alpha} \colon S \to I_{E(S)}\) to be the homomorphism associated to $\alpha$. 
In \cref{lem:preaction_intro} we will show that $T = S / \rho$ has a preaction (\cref{dfn-preaction}) on $\mathcal{Y} = E(S)$ (the is essentially the same idea as \cref{prop:comp-rel}). Hence by \cref{thm:extend_preactions} there will exist an inverse semigroup action of $T$ on a poset $\mathcal{X}$ containing $\mathcal{Y}$. 

\begin{dfn}\label{defn:doma_product}
Let \(\alpha\) be the action given in \eqref{eq-action}.
We define a multiplication on the set \(\dom(\alpha)\) by
    \[(e, s)(f, t)=(((e, s)\alpha\wedge f, s^{-1})\alpha, st).\]
   This multiplication is well-defined as \(((e, s)\alpha\wedge f, s^{-1})\alpha \leq sfs^{-1}\leq stt^{-1}s^{-1}=(st)(st)^{-1}\). We do not assert that this multiplication is associative. 
\end{dfn}

The natural magma homomorphism 
\begin{equation}\label{eq-pi}
\pi\colon \dom(\alpha)\to S \text{ is defined by }
(e, s)\pi = s.
\end{equation}
We will show that a subset of $\dom(\alpha)$ with the multiplication given in \cref{defn:doma_product} is a semigroup, by showing that the subset is (magma) isomorphic to a semigroup. 

\begin{lem}\label{lem:psi_hom}
    Let \(\psi \colon S \to \dom(\alpha)\) be the map defined by
    \[(s)\psi=(ss^{-1}, s),\]
    Then \(\psi\) is an injective magma homomorphism and \(\im(\psi)\cong S\). 
\end{lem}
\begin{proof}
    Note that $\psi$ is well-defined by the definition of the set \(\dom(\alpha)\). 
    
    Let \(s,t\in S\). Then
    \begin{align*}
        (s)\psi\ (t)\psi & =(ss^{-1}, s)(tt^{-1}, t) \\
        & = ((s^{-1}ss^{-1}s \wedge tt^{-1}, s^{-1})\alpha,
        st) \\
        & =((s^{-1}s \wedge tt^{-1}, s^{-1})\alpha, st) \\
        & = (ss^{-1}stt^{-1}s^{-1},st) \\
        & = (stt^{-1}s^{-1}, st) \\ 
        & = (st)\psi.
    \end{align*}
    Since the homomorphism \(\psi\circ \pi\) is the identity map on \(S\), the restriction \(\pi |_{\im(\psi)} = \psi ^ {-1}\) is an isomorphism from a subsemigroup of \(\dom(\alpha)\) to \(S\).
\end{proof}

We have not yet defined the poset \(\mathcal X\) which we will be using to define our \(Q\)-semigroup. However, since the set of elements of the \(Q\)-semigroup do not depend on
\(\mathcal X\), only the multiplication within the \(Q\)-semigroup. The next lemma shows that $S$ is contained in the set of elements in the $Q$-semigroup we are in the process of defining. 

\begin{lem}
    \label{dfn:M}
    Let \(\rho\) be the syntactic congruence on \(S\), and let \(\psi_{\rho} \colon S \to E(S)\times S/\rho\) by
    \[(s)\psi_{\rho}=(ss^{-1}, s/\rho).\]
    Then  \(\psi_\rho\) is injective. 
\end{lem}

\begin{proof}
    Let \(s, t \in S\) be such that \((s)\psi_{\rho}= (ss^{-1}, s / \rho) = (tt^{-1}, t / \rho) = (t)\psi_{\rho}\). In particular, \(ss^{-1} = tt^{-1}\) and so \(s \mathscr{R} t\). Since the quotient homomorphism from $S$ to $S/\rho$ is idempotent-pure, \cref{injR} implies that this homomorphism is injective on the
    \(\mathscr{R}\)-classes of \(S\). Hence 
     \(s / \rho = t / \rho\) implies that  \(s = t\).
\end{proof}
    
    We define 
    \begin{equation}\label{eq-the-def-of-M}
    M = \im(\psi_\rho) = \{(ss^{-1}, s/\rho) : s \in S\} \subseteq E(S)\times S/\rho.
    \end{equation}
   and define multiplication on $M$ such that \(\psi_\rho \colon S\to M\) is an isomorphism.
If $\psi \colon S \to \im(\psi) \subseteq \dom(\alpha)$ is the (semigroup) isomorphism from \cref{lem:psi_hom} and $\pi \colon \dom(\alpha) \to S$ is from \eqref{eq-pi}, then  
 \(\pi|_{\im(\psi)}\psi_\rho = \psi^{-1}\psi_\rho \colon \im(\psi)\to M\) is an isomorphism.
If $(ss^{-1}, s)\in \im(\psi)$, then 
    \[(ss^{-1}, s)\psi^{-1}\psi_\rho= (s) \psi_{\rho} = (ss^{-1}, s/\rho).\]

If $s, t\in S$ and $e, f\in E(S)$ are such that $(e, s/\rho), (f, t/\rho)\in M$, then  
there is \(s_0\in s/\rho\)  such that \((s_0)\psi_\rho=(s_0s_0^{-1}, s_0/\rho)=(e, s/\rho)\). 
Since 
\begin{align*}
(e, s/\rho)(f, t/\rho) &=((e, s)(f, t))\psi^{-1}\psi_\rho  \\
& = 
((((e, s_0)\alpha\wedge f), s_0^{-1})\alpha, st)\psi^{-1}\psi_\rho \\
& = 
((((e, s_0)\alpha\wedge f), s_0^{-1})\alpha, st/\rho),
\end{align*}
it follows that
\begin{equation}\label{eq-Q-mult}
(e, s/\rho)(f, t/\rho) = 
((((e, s_0)\alpha\wedge f), s_0^{-1})\alpha, st/\rho).
\end{equation}
We will prove \cref{thm:Q-thm} by describing the multiplication of the given inverse semigroup \(S\) using only the $E$-disjunctive inverse semigroup \(S/\rho\), the idempotents $E(S)$, and an action of $S/\rho$ on a poset.  
Since \(M\) is isomorphic to \(S\), the above equation almost does this. The problem is that $\alpha$ is defined in terms of $S$, and not only in terms of $S/\rho$ and $E(S)$.
We will show that \(S/\rho\) is sufficient to capture the needed information from this action.

Since the particular choice of representative of the classes in $S/\rho$ is not important later, we denote $S/ \rho$ by $T$ so that we may refer to the elements of $T$ rather than choosing a representative for an element of $S/\rho$.

\begin{lem}\label{lem:preaction_intro}
    If $\alpha \colon E(S) \times S \to E(S)$ is the action defined in \eqref{eq-action}, then the partial function \(q \colon E(S) \times T \to E(S)\) defined by
    \[
    (e, t)q = (e, s)\alpha
    \]
    for all $(e, t) \in E(S) \times T$  such that there exists $t' \leq t$ with
    $s\in t'$ is a preaction. In particular, 
    \[
    \dom(q)  = \{(e, t) \in E(S)\times S/\rho  \mid \exists s\in t'\leq t\text{ with }(e, s)\in \dom(\alpha)\}.
    \]
\end{lem}
\begin{proof}
     We first show that \(q\) is well-defined.
     Let \(s_1\in t_1\leq t\in T\), let \(s_2\in t_2\leq t\in T\), and let \((e,s_1), (f, s_2)\in \dom(\alpha)\). We will show that \((e, s_1)\alpha\leq (f, s_2)\alpha\) if and only if \(e\leq f\).
     This will not only show that \(q\) is well-defined (by considering the case when \(e=f\)) but will also show it satisfies \cref{dfn-preaction}\eqref{defn:preaction_preservesorder}.

     Since $s_1\in t_1\leq t$ and $s_2\in t_2\leq t$, it follows that \((s_1^{-1}s_2/\rho)\leq t^{-1}t\) and $s_1s_2^{-1}/\rho \leq t t ^{-1}$. Hence $s_1^{-1}s_2/\rho, s_1s_2^{-1}/\rho \in E(T)$, and so, 
     by Lallement's Lemma, both $s_1^{-1}s_2/\rho$ and $s_1s_2^{-1}/\rho$ contain an idempotent. 
     But \(\rho\) is idempotent-pure and so \(s_1^{-1}s_2, s_1s_2^{-1} \in E(S)\). Thus \(s_1^{-1} s_2\) and $s_1s_2^{-1}$ equal their inverses, that is, 
     \begin{equation}\label{eq-obvious-0}
     s_1^{-1}s_2 = s_2^{-1} s_1\quad\text{and}\quad s_1s_2^{-1}= s_2s_1^{-1}
     \end{equation}
     Similarly \(s_1s_2^{-1}=s_2s_1^{-1}\) is an idempotent. 
     If $s\in S$ we define $\underline{s}_{\alpha} \colon E(S) \to E(S)$, by $(g)\underline{s}_{\alpha} = (g, s)\alpha = s^{-1} g s$ for all $g\in E(S)$, then 
     $\dom(\underline{s}_{\alpha})$ is an order ideal. 
     This notation coincides with the notation in \cref{dfn-preaction} although we have not shown that $\alpha$ is a preaction.
     We assumed at the start of the proof that \((e,s_1), (f, s_2)\in \dom(\alpha)\) and so
\((e,s_1)\alpha\in \im(\underline{s_1}_\alpha)\) and \(e\in \dom(\underline{s_2}_\alpha)\) since \(s_1\leq s_2\) and \(e\leq f\). 
In particular, $(e, s_1)\alpha \in \dom(\underline{s_1}_{\alpha}^{-1})$, and so
$(e, s_1)\alpha \underline{s_1}_{\alpha}^{-1}= (e) \underline{s_1}_\alpha\underline{s_1}_{\alpha}^{-1} = e$.
Since $e\in \dom(\underline{s_{2}}_\alpha)$, it follows that 
$(e, s_1)\alpha \underline{s_1}_{\alpha}^{-1}\in \dom(\underline{s_2}_{\alpha})$
and so $(e, s_1)\alpha \in \dom(\underline{s_1^{-1}s_2}_{\alpha})$.
Since $s_1^{-1}s_2$ is an idempotent, $(s_1^{-1}s_2)_{\alpha}$ acts as the identity on every point in its domain, including $(e, s_1)\alpha$. In other words, 
\begin{equation}\label{eq-obvious-1}
(e, s_1)\alpha = (e, s_1s_1^{-1}s_2)\alpha. 
\end{equation}

By definition, \((e, s_1s_1^{-1}s_2)\alpha=(e, s_1s_1^{-1})\alpha\cdot \underline{s_2}_\alpha\).
Since $\underline{s_1s_1^{-1}}_{\alpha}$ is the identity $\dom(\underline{s_1}_{\alpha})$ and $e\in \dom(\underline{s_1}_{\alpha})$, it follows that 
\begin{equation}\label{eq-obvious-2}
(e, s_1s_1^{-1}s_2)\alpha = (e)\underline{s_1s_1^{-1}s_2}_\alpha = (e)\underline{s_1s_1^{-1}}_\alpha\circ\underline{s_2}_\alpha
= (e)\underline{s_2}_{\alpha} = (e, s_2)\alpha.
\end{equation}

 Therefore if \(s_1\in t_1\leq t\in T\), \(s_2\in t_2\leq t\in T\), and \((e,s_1), (f, s_2)\in \dom(\alpha)\), then 
     \begin{align*}
         e\leq f &\Rightarrow e\leq f \text{ and }(e, s_1)\alpha = (e, s_1s_1^{-1}s_2) \alpha = (e, s_2) \alpha && \text{by }\eqref{eq-obvious-1} \text{ and } \eqref{eq-obvious-2}\\
         &\Rightarrow (e, s_1) \alpha = (e, s_2) \alpha \leq (f, s_2) \alpha && \underline{s_2}_{\alpha}\text{ is an order isomorphism and }e\leq f\\
        &\Rightarrow (e, s_1) \alpha \leq (f, s_2) \alpha\\
        &\Rightarrow ((e, s_1) \alpha, s_2^{-1}) \alpha \leq ((f, s_2) \alpha, s_2^{-1}) \alpha\\
         &\Rightarrow e = (e, s_1s_2^{-1}) \alpha \leq (f, s_2s_2^{-1})\alpha = f &&s_1s_2^{-1}\in E(S)\text{ by \eqref{eq-obvious-0} and } e\in \dom(\underline{s_1s_2^{-1}}_{\alpha})
         \\
          &\Rightarrow e\leq f.
     \end{align*}
 
    That \cref{dfn-preaction}\eqref{dfn:preaction_downwards_agreement}, \eqref{defn:down_closed condition} and \eqref{defn:preaction_uses_Y} hold is clear. 
    The remaining condition is condition \eqref{dfn:preaction_connnected_groupoid*}.
    Suppose that $t_1, t_2\in T$ and $(e)\underline{t_1}_q =f$ and $(f)\underline{t_2}_q = g$. We must show that  
    $(e)\underline{t_1t_2}_q =  g$ and $(f)\underline{t_1^{-1}}_q = e$.

    We first want to show that \(e \in \dom(\underline{t_1t_2}_q)\). Since
    \(e \in \dom(\underline{t_1}_q)\), from the definition of \(q\) 
    there exists \(s_1 \in t_1' \leq t_1\) such that \((e, s_1) \in
    \dom(\alpha)\). Similarly, there exists \(s_2 \in t_2' \leq t_2\) with \((f, s_2) \in \dom(\alpha)\). We want to show that there exists
    \(s_3 \in t_3' \leq t_1 t_2\) with \((e, s_3) \in \dom(\alpha)\).
    
    Set \(s_3 = s_1 s_2\). By assumption \((e, s_1), (f, s_2) \in \dom(\alpha)\),
    and so 
    \((e, s_1)\alpha = (e) \underline{s_1}_\alpha =(e, s_1)\alpha = (e, t_1)q = f\)
    and, similarly,  \((f, s_2)\alpha = g\).
    Thus, since $\alpha$ is an action, $g = (f, s_2)\alpha = ((e, s_1)\alpha, s_2)\alpha = (e, s_1s_2)\alpha$. In particular, 
    \((e, s_1 s_2) \in \dom(\alpha)\). 
    Since $s_1\in t_1'\leq t_1$ and $s_2\in t_2'\leq t_2$, it follows that
    \(s_1 s_2 \in t_1' t_2' \leq t_1t_2\) and so
    \[
    (e)\underline{t_1t_2}_q = (e, t_1t_2)q 
                            = (e, s_1s_2)\alpha 
                            = (f, s_2)\alpha
                            = g,
    \]
    as required.

    It remains to show that $(f)\underline{t_1^{-1}}_q = e$. 
    Again,
    we begin by showing that \(f \in \dom(\underline{t_1^{-1}}_q)\).
    Since $(e, s_1)\in \dom(\alpha)$, there exists \(s_1 \in t_1' \leq t_1\) such that
    \((e, s_1) \in \dom(\alpha)\). Since \(\alpha\) is an inverse semigroup
    action, \(((e)\underline{s_1}_\alpha, s_1^{-1}) \in \dom(\alpha)\).
    As we showed earlier, \((e)\underline{s_1}_\alpha = f\), and so
    \((f, s_1^{-1}) \in \dom(\alpha)\) and \((f, t_1^{-1}) \in \dom(q)\).
    Thus \((f)\underline{t_1^{-1}}_q = e\),
    as required.\qedhere
\end{proof}

We can now prove \cref{thm:Q-thm}.  

\begin{theorem_no_number}[\ref{thm:Q-thm}]
Every inverse semigroup \(S\) is isomorphic to some $Q(T, \mathcal{Y}, \mathcal{X})$ from \cref{dfn-q-semigroup}, where $T$ is the maximum $E$-disjunctive homomorphic image of \(S\), and \(\mathcal{Y}\) is the semilattice of idempotents of \(S\).
\end{theorem_no_number}

\begin{proof}
Let \(S\) be any inverse semigroup; let \(\mathcal{Y}=E(S)\);
let $\rho$ be the syntactic congruence on $S$;
let $T = S/ \rho$; and 
let $\phi_{\rho} \colon S \to T$ be the natural homomorphism defined by $(s)\phi_{\rho} = s/\rho$.
We also recall the following:
\begin{itemize}
\item let $\alpha \colon E(S) \times S \to E(S)$ be the inverse semigroup action defined by $(e, s) \alpha = s ^ {-1}es$ (see \eqref{eq-action})
for all $(e, s) \in E(S) \times S$ such that $e\leq ss ^ {-1}$;
\item let $\pi \colon \dom(\alpha) \to S$ be defined by $(e, s)\pi = s$ (see \eqref{eq-pi}); 
\item let $\psi \colon S \to \dom(\alpha)$  be defined by 
$(s)\psi = (ss ^ {-1}, s)$ (\cref{lem:psi_hom});
\item let
$M = \set{(ss^{-1}, s/\rho)\in E(S)\times T}{s\in S}$ as defined in \eqref{eq-the-def-of-M};
\item let $q \colon E(S)\times T \to T$ be the preaction defined in \cref{lem:preaction_intro};
\item  let \(\mathcal{X}\) be the poset defined in  \cref{thm:extend_preactions} (with respect to the preaction $q$) that contains $\mathcal{Y}$;
\item let $\beta\colon \mathcal{X}\times S \to \mathcal{X}$ be the action given in \cref{thm:extend_preactions} such that $\beta$ restricted to $(\mathcal{Y}\times S)\cap (\mathcal{Y})\beta ^{-1}$ equals $q$.
\end{itemize}

We will verify that
$T$, $\mathcal{Y}$, and $\mathcal{X}$ satisfy the conditions in \cref{dfn-q-semigroup}. Firstly, by \cref{thm:extend_preactions}, $\beta$ is an inverse semigroup action of \(T\) on \(\mathcal X\) by partial order isomorphisms, \(\mathcal Y\) is a meet
subsemilattice, and order ideal, of $\mathcal{X}$.

\begin{enumerate}[(1)]
    \item 
        Let $t\in T$ be arbitrary. We must show that $\dom \underline{t}_{\beta}$ is an order ideal of $\mathcal{X}$. We can assume without loss of generality that \(t\) is an idempotent because \(\dom(\underline{t}_\beta)=\dom(\underline{tt^{-1}}_\beta)\) for all $t\in T$.
        Let $a\in\mathcal{X} \) and \(b \in \dom \underline{t}_\beta\subseteq \mathcal{X}\) be such that
        \(a \leq b\). By \cref{thm:extend_preactions}(3), there exists $s\in T^1$ such that
        \((a, s) \beta, (b, s) \beta \in \mathcal{Y}\) and \((a, s) \beta \leq (b, s) \beta\). In other words, 
        $(a)\underline{s}_{\beta} \leq (b)\underline{s}_{\beta}$. 
        Since \(\dom(\underline{s^{-1}ts}_q)\) is an order ideal in \(\mathcal{Y}\) (by \cref{dfn-preaction}(1))  and \(\mathcal{Y}\) is an order
        ideal in \(\mathcal X\), \(\dom(\underline{s^{-1}ts}_q) \cap \mathcal{Y}\) is an order ideal in \(\mathcal{X}\). By assumption $t$ is an idempotent, and so \(s^{-1}ts\) is an idempotent also. It follows that \(\dom(\underline{s^{-1}t}_\beta) \cap \mathcal{Y}=\dom(\underline{s^{-1}ts}_\beta) \cap \mathcal{Y}=\dom(\underline{s^{-1}ts}_q) \cap \mathcal{Y}\) is an order ideal of \(\mathcal{X}\). 
        
        By assumption \(b \in \dom \underline{t}_\beta\) and so \((b)\underline{s}_{\beta} \in \dom(\underline{s^{-1}t}_\beta) \cap \mathcal{Y}\). Finally, since $\dom(\underline{s^{-1}t}_\beta) \cap \mathcal{Y}$ is an order ideal,  \((a)\underline{s}_{\beta} \in \dom(\underline{s^{-1}t}_\beta)\), and so \(a \in \dom(\underline{t}_\beta)\), as required.
    \item 
    Firstly,  for all \((y, t)\in \mathcal{Y}\times T\), the following hold: 
    \begin{align}\label{eq-iff}
    \begin{split}
    y\in \dom(\underline{t}_{\beta}) & \iff y\in \dom(\underline{tt^{-1}}_\beta) \\
    & \iff y\in \dom(\underline{tt^{-1}}_q) \\
    & \iff \text{there exists }e\in S \text{ such that } e/\rho\leq tt^{-1} \text{ and }y\leq e \\
    & \iff y/\rho \leq tt^{-1}.
    \end{split}
    \end{align}
    So if \(y\in \mathcal{Y}\), then setting $t = y / \rho\in T$ and repeatedly applying \eqref{eq-iff} we obtain
    \begin{align*}
    \delta(y) = \dom(\underline{t}_\beta|_{\mathcal{Y}}) 
    &= \{z \in \mathcal{Y} \mid z / \rho \leq tt^{-1}\} \\
    &= \bigcap\set{\dom(\underline{t_1}_\beta|_{\mathcal{Y}})}{t_1\in T
\text{ such that } y / \rho \leq t_1t_1^{-1}}\\
    & =
    \bigcap\set{\dom(\underline{t_1}_\beta|_{\mathcal{Y}})}{t_1\in T
\text{ such that } y\in \dom(\underline{t_1}_\beta|_{\mathcal{Y}})}.
    \end{align*}
    \item We must show that for all \(x\in \mathcal{X}\), there is \(t\in T\) such that \((x,t)\beta\in \mathcal{Y}\). This is implied by~\cref{thm:extend_preactions}(3).
\end{enumerate}

Since $S$ and $M$ are isomorphic, by the definition of the multiplication of $M$, it suffices to show that $M$ and $Q(T, \mathcal{Y}, \mathcal{X})$ coincide (as semigroups).

The following holds:
\begin{align*}
    Q(T, \mathcal{Y}, \mathcal{X}) 
     & = \makeset{(y,t)\in \mathcal{Y}\times T}{\(\dom (\underline{t}_{\beta}) = \Dom(y) = \dom(\underline{y/\rho}_\beta|_{\mathcal{Y}}),\ (y, t)\beta \in \mathcal{Y}\)}\\
    &=\makeset{(y, s/\rho)\in \mathcal{Y}\times T}{\(\dom(\underline{s/\rho}_\beta) \cap \mathcal{Y}= \dom(\underline{y/\rho}_\beta)\cap \mathcal{Y}\), \((y, s/\rho)\beta \in \mathcal{Y}\)}\\
     &=\makeset{(y, s/\rho)\in \mathcal{Y}\times T}{\( \dom(\underline{ss^{-1}/\rho}_\beta) \cap \mathcal{Y}= \dom(\underline{y/\rho}_\beta)\cap \mathcal{Y}\), \((y, s/\rho)\beta \in \mathcal{Y}\)}\\
       &=\makeset{(y, s/\rho)\in \mathcal{Y}\times T}{\( \bigcup (ss^{-1}/\rho)\downset = \bigcup (y/\rho)\downset\subseteq\mathcal{Y}\), \((y, s/\rho)\beta \in \mathcal{Y}\)} && (\text{by } \eqref{eq-iff})\\
      &=\makeset{(y, s/\rho)\in \mathcal{Y}\times T}{\(  (ss^{-1}/\rho)\downset =  (y/\rho)\downset\), \((y, s/\rho)\beta \in \mathcal{Y}\)}\\
    &=\makeset{(y, s/\rho)\in \mathcal{Y}\times T}{\(  (ss^{-1}, y)\in \rho\), \((y, s/\rho)\beta \in \mathcal{Y}\)}\\
    &=\makeset{(y, s/\rho)\in \mathcal{Y}\times T}{\(  (ss^{-1}, y)\in \rho\), \((y, s/\rho)q \in \mathcal{Y}\)}\\
    &=\makeset{(y, s/\rho)\in \mathcal{Y}\times T}{\(  (ss^{-1}, y)\in \rho\), \((y, s/\rho)\in \dom(q)\)}\\
        &=\makeset{(y, s/\rho)\in \mathcal{Y}\times T}{\(  (ss^{-1}, y)\in \rho, \exists a\in S\) with \(a/\rho\leq s/\rho\) and \(y\leq aa^{-1}\)} && (\text{def. of }q)\\
        & = Q_1.
    \end{align*}
If $(y, s'/\rho)\in Q_1$, then there exists $s\in s'/\rho$ such that \((ss^{-1}, y)\in \rho\) and there exists \(a\in S\) with \(a/\rho\leq s/\rho\) and \(y\leq aa^{-1}\). Thus \(ss^{-1}/\rho=y/\rho\leq aa^{-1}/\rho\). 
Moreover
\[ss^{-1}/\rho\leq aa^{-1}/\rho\iff ss^{-1}a/\rho\leq aa^{-1}a/\rho\iff s/\rho\leq a/\rho \iff s/\rho= a/\rho .\]
Thus
    \begin{align*}
        Q_1             &=\makeset{(y, s/\rho)\in \mathcal{Y}\times T}{\(  (ss^{-1}, y)\in \rho, \exists a\in S\) with \(a/\rho= s/\rho\) and \(y\leq aa^{-1}\)}\\
            &=\makeset{(y, s/\rho)\in \mathcal{Y}\times T}{\(  (ss^{-1}, y)\in \rho\), and \(y\leq ss^{-1}\)}\\
      &=\makeset{(y, s/\rho)\in \mathcal{Y}\times T}{\(  (ss^{-1}, y)\in \rho\), \(y\leq ss^{-1}\)} && \\
        &=\makeset{(y, s/\rho)\in \mathcal{Y}\times T}{\(  y= ss^{-1}\)} \\
    &=\makeset{(ss^{-1}, s/\rho)\in \mathcal{Y}\times T}{\(s\in S\)}=M.
\end{align*}
We have shown that $M$ and $Q(T, \mathcal{Y}, \mathcal{X})$ are equal as sets. That their multiplications also coincide is precisely \eqref{eq-Q-mult}, and so the proof is complete.
\end{proof}

\section*{Acknowledgements}
The authors would like to thank James East and Peter Hines for some helpful mathematical discussions.
The authors were supported by a Heilbronn Institute for Mathematical Research Small Grant during part of this work. The authors would like to thank the anonymous referee for their helpful comments and careful reading of the paper. The second named author was supported by a London Mathematical Society Early Career Fellowship and the Heilbronn Institute for Mathematical Research during this work. The authors would also like to thank the University of Manchester for hosting them during part of the work on this paper. 

\printbibliography
\end{document}